\documentclass[11pt,a4paper]{article}

\usepackage{a4wide}
\usepackage{amsmath}
\usepackage{amssymb}
\usepackage{booktabs}       % professional-quality tables
\usepackage{algorithm}
\usepackage{algorithmic}
\usepackage{bm}
\usepackage{lscape}

\usepackage[utf8]{inputenc} % allow utf-8 input
\usepackage[T1]{fontenc}    % use 8-bit T1 fonts
\usepackage{hyperref}       % hyperlinks

\usepackage{amsfonts}       % blackboard math symbols
\usepackage{graphicx}       % graphics
\usepackage{cleveref}
\usepackage{amsthm}
\usepackage{url}            % simple URL typesetting
\usepackage{cite}
\usepackage{mathrsfs}

\usepackage{nicefrac}       % compact symbols for 1/2, etc.
\usepackage{microtype}      % microtypography
\usepackage{fancyhdr}       % header

\DeclareMathOperator{\Exp}{Exp}
\DeclareMathOperator{\grad}{grad}

\DeclareMathOperator{\sgn}{sgn}
\DeclareMathOperator{\image}{Im}

\def\D{\mathrm{D}}
\def\fmincon{\mathrm{fmincon}}
\def\id{\mathrm{id}}
\def\Lasso{\mathrm{Lasso}}
\def\nonreg{\mathrm{nonreg}}
\def\proj{\mathrm{proj}}
\def\proposed{\mathrm{proposed}}
\def\orth{\mathrm{orth}}
\def\unconst{\mathrm{unconst}}

\newcommand{\argmin}{\mathop{\rm arg~min}\limits}

\newtheorem{proposition}{Proposition}[section]
\newtheorem{theorem}{Theorem}[section]

\theoremstyle{remark}

\newtheorem{remark}{Remark}[section]

\crefname{proposition}{Proposition}{Propositions}
\crefname{theorem}{Theorem}{Theorems}
\crefname{corollary}{Corollary}{Corollaries}
\crefname{lemma}{Lemma}{Lemmas}
\crefname{definition}{Definition}{Definitions}
\crefname{remark}{Remark}{Remarks}
\crefname{assumption}{Assumption}{Assumptions}
\crefname{example}{Example}{Examples}
\crefname{problem}{Problem}{Problems}
\crefname{remark}{Remark}{Remarks}
\crefname{section}{Section}{Sections}
\crefname{subsection}{Section}{Sections}
\crefname{subsubsection}{Section}{Sections}
\crefname{figure}{Figure}{Figures}

\title{Riemannian optimization on unit sphere with $p$-norm\\ and its applications\footnotetext{\textbf{Funding:} This work was funded by JSPS KAKENHI Grant number JP20K14359.}}

\author{Hiroyuki Sato\thanks{Department of Applied Mathematics and Physics, Kyoto University, Kyoto, Japan\newline ({\tt hsato@i.kyoto-u.ac.jp}).}}

\begin{document}
\maketitle

\begin{abstract}
This paper deals with Riemannian optimization on the unit sphere in terms of $p$-norm with general $p > 1$. As a Riemannian submanifold of the Euclidean space, the geometry of the sphere with~$p$-norm is investigated, and several geometric tools used for Riemannian optimization, such as retractions and vector transports, are proposed and analyzed. Applications to Riemannian optimization on the sphere with nonnegative constraints and $L_p$-regularization-related optimization are also discussed. As practical examples, the former includes nonnegative principal component analysis and the latter is closely related to the Lasso regression and box-constrained problems. Numerical experiments verify that Riemannian optimization on the sphere with $p$-norm has substantial potential for such applications, and the proposed framework provides a theoretical basis for such optimization.
\end{abstract}


\section{Introduction}
\label{sec1}
In the Euclidean space $\mathbb{R}^n$, the $p$-norm of a vector $a \in \mathbb{R}^n$ whose $i$th element is $a_i \in \mathbb{R}$ is defined by
\begin{equation}
     \|a\|_p := \sqrt[p]{\sum_{i=1}^n \lvert a_i\rvert^p},
\end{equation}
where $p \ge 1$ is a real value.
When $p = \infty$, the $\infty$-norm, or maximum norm, is defined by
\begin{equation}
     \|a\|_\infty := \max\left\{\lvert a_1\rvert, \lvert a_2\rvert, \dots, \lvert a_n\rvert\right\}.
\end{equation}
In optimization and related fields, discussions are usually based on the $2$-norm.
The $1$-norm is also important in, e.g., Lasso regression for sparse estimation~\cite{hastie2015statistical}.
Furthermore, for $x \in \mathbb{R}^n$, the constraint $\|x\|_\infty \le c$ for some $c \geq 0$ is equivalent to the box constraint $-c \le x_i \le c$ for all elements $x_i$ of $x$.

For $p \ge 1 $ or $p = \infty$, we define the unit sphere with $p$-norm in $\mathbb{R}^n$ as
\begin{equation}
    S^{n-1}_p := \{x \in \mathbb{R}^n \mid \|x\|_p = 1\}.
\end{equation}
A particularly important and well-studied example is the case of $p = 2$, which reduces to the standard (hyper)sphere $S^{n-1}_2 = \{x \in \mathbb{R}^n \mid \|x\|_2 = 1\}$ in the sense of the Euclidean norm.
In terms of optimization, as we discuss in Section~\ref{sec:application}, the case of $p = 2p'$ can be used to implicitly impose the nonnegativity constraints on $x \in S^{n-1}_{p'}$.
A practical example of this is the case of~$p = 4$ and $p' = 2$, which leads to a constrained optimization on the standard unit sphere $S^2_2$ with the constraint $x \ge 0$.
Furthermore, the case of $p = 1$ is closely related to $L_1$ regularization in, e.g., Lasso~\cite{hastie2015statistical}, and the case of $p=\infty$ is closely related to the box constraint.

In this paper, we address the geometry of $S^{n-1}_p$ with $p \in (0, \infty)$ and provide several mathematical tools required for Riemannian optimization, i.e., optimization on Riemannian manifolds, such as retractions and vector transports~\cite{AbsMahSep2008,sato2021riemannian}.
A natural and practical retraction is defined through normalization in terms of $p$-norm, and we provide mathematical support for the validity of this retraction.
Furthermore, we discuss projective and orthographic retractions on~$S^{n-1}_p$.
Although it may be harder to use such retractions practically than the retraction based on normalization, their inverses are efficient and easy to implement.
Thus, discussing them is meaningful.
We also provide an explicit expression for the vector transport defined as the differentiated retraction associated with the retraction by normalization.
Other contributions of this paper include applications of the sphere $S^{n-1}_p$ to practical optimization problems related to, e.g., the nonnegative principal component analysis (PCA) and Lasso regression.

This paper is organized as follows.
In Section~\ref{sec:preliminary}, we introduce the notations used.
We also review the differentiability and derivative of the $p$-norm, which are used throughout this paper.
In Section~\ref{sec:geometry}, we prove that $S^{n-1}_p$ is a Riemannian submanifold of $\mathbb{R}^n$ and, as such, investigate its geometry.
Section~\ref{sec:retraction} provides a retraction on $S^{n-1}_p$ based on normalization and its inverse.
The respective formulas for the inverses of projective and orthographic retractions are also provided.
In Section~\ref{sec:VT}, we discuss a vector transport on $S^{n-1}_p$ derived by differentiating a retraction.
We also remark another vector transport based on the orthogonal projection.
Section~\ref{sec:summary} is a reference to the geometric results in this paper.
We present two types of applications of Riemannian optimization on $S^{n-1}_p$ in Section~\ref{sec:application}.
One is the application to Riemannian optimization problems on the sphere with the nonnegative constraint, which include nonnegative PCA as an important example.
The other is the application to $L_p$-regularization-related optimization problems, which include the Lasso regression and box-constrained problems.
Section~\ref{sec:conclusion} concludes the paper.

\section{Preliminaries}
\label{sec:preliminary}
In this section, we provide preliminaries for the discussion in the later sections.

\subsection{Notation}
Throughout the paper, we use the following notation.
The vector space of $n$-dimensional real column vectors is denoted by $\mathbb{R}^n$.
We use the notation $\cdot^T$ to indicate transposition.
The~$n$-dimensional real vector whose $i$th element is~$a_i \in \mathbb{R}$ is denoted by $(a_i) \in \mathbb{R}^n$, and we denote the $i$th element of $b \in \mathbb{R}^n$ by $b_i$ or $(b)_i$.
For $a = (a_i) \in \mathbb{R}^n$, we denote the element-wise power of $r \in \mathbb{R}$ by~\mbox{$a^r := (a_i^r) \in \mathbb{R}^n$} and the element-wise absolute value by~\mbox{$\lvert a\rvert := (\lvert a_i\rvert) \in \mathbb{R}^n$}.
Furthermore, the binary relation $\le$ (resp. $\ge$) for vectors $a = (a_i), b = (b_i) \in \mathbb{R}^n$ means the element-wise relation $\le$ (resp. $\ge$), i.e., $a \le b$ (resp. $a \ge b$) is equivalent to $a_i \le b_i$ (resp.~$a_i \ge b_i$) for $i = 1, 2, \dots, n$.
In particular, $a \ge 0$ means that all elements of $a$ are nonnegative.
We define the all-one vector as~$\bm{1} := (1, 1, \dots, 1)^T \in \mathbb{R}^n$. Then, the condition $\|x\|_p = 1$ is equivalent to~$\|x\|_p^p = 1$ and rewritten as $\bm{1}^T \lvert x\rvert^p = 1$.
The identity matrix of $n$th order is denoted by $I$.
For $\bm{1} \in \mathbb{R}^n$ and $I \in \mathbb{R}^{n \times n}$, the size $n$ is determined by context.

We denote the sign function by $\sgn$, i.e.,
\begin{equation}
    \sgn(w) := \begin{cases}1 \quad \text{if $w > 0$},\\ 0  \quad \text{if $w = 0$},\\ -1 \quad \text{if $w < 0$}\end{cases}
\end{equation}
for $w \in \mathbb{R}$.
Note that $\sgn(w)\lvert w\rvert = w$ always holds.
We also use the same notation for the element-wise application of $\sgn$, i.e., for $a = (a_i) \in \mathbb{R}^n$, we define $\sgn(a) := (\sgn(a_i)) \in \mathbb{R}^n$.

The operator $\odot$ denotes the Hadamard product, which is the element-wise product, i.e., for~$a = (a_i), b = (b_i) \in \mathbb{R}^n$, we define $a \odot b := (a_i b_i) \in \mathbb{R}^n$.
We consider the Hadamard product only for vectors in this paper.
It is clear that the commutative law~\mbox{$a \odot b = b \odot a$} holds.
Furthermore, for $c = (c_i) \in \mathbb{R}^n$, we have $a^T (b \odot c) = (a \odot b)^T c$ because both sides are equal to~$\sum_{i=1}^n a_ib_ic_i$.
Using these facts, we can rewrite the condition $\|x\|_p^p = 1$ as~\mbox{$x^T (\sgn(x) \odot \lvert x \rvert^{p-1}) = 1$} because we have
\begin{equation}
    x^T (\sgn(x) \odot \lvert x\rvert^{p-1}) = (\sgn(x) \odot x)^T \lvert x\rvert^{p-1} = \lvert x\rvert^T \lvert x \rvert^{p-1} = \bm{1}^T \lvert x\vert^p = \|x\|_p^p.
\end{equation}

Although $\mathbb{R}^n$ can be equipped with the $p$-norm to be a normed vector space, no inner product is associated with the $p$-norm unless $n = 1$ or $p = 2$.
Therefore, we equip $\mathbb{R}^n$ with the standard inner product $\langle a, b\rangle := a^T b$ and the induced norm $\|a\| := \sqrt{\langle a, a\rangle} = \|a\|_2$, which coincides with the $2$-norm, even when we discuss the sphere $S^{n-1}_p$ for general $p$.
As discussed in Section~\ref{sec:geometry}, we regard $\mathbb{R}^n$ as a Riemannian manifold with the Riemannian metric induced by the standard inner product and consider $S^{n-1}_p$ for $p \in (1, \infty)$ as a Riemannian submanifold of~$\mathbb{R}^n$.

For a manifold $\mathcal{M}$, we denote the tangent space of $\mathcal{M}$ at $x \in \mathcal{M}$ by~$T_x \mathcal{M}$.
Furthermore, when the manifold $\mathcal{M}$ is a Riemannian manifold with a Riemannian metric $\langle \cdot, \cdot\rangle$, each tangent space $T_x \mathcal{M}$ is endowed with the inner product $\langle \cdot, \cdot \rangle_x$ via the Riemannian metric~$\langle \cdot, \cdot\rangle$, and the Riemannian gradient $\grad f(x)$ of a $C^1$ function $f \colon \mathcal{M} \to \mathbb{R}$ at $x$ is defined as the unique tangent vector at $x$ satisfying $\D f(x)[\xi] = \langle \grad f(x), \xi\rangle_x$ for all $\xi \in T_x \mathcal{M}$, where~$\D f(x) \colon T_x \mathcal{M} \to T_{f(x)}\mathbb{R} \simeq \mathbb{R}$ is the derivative of $f$ at $x \in \mathcal{M}$.
For~$\mathbb{R}^n$ as a Riemannian manifold with the Riemannian metric $\langle \xi, \eta\rangle_x := \xi^T \eta$ for any $x \in \mathbb{R}^n$ and $\xi,\, \eta \in T_x \mathbb{R}^n \simeq \mathbb{R}^n$, the Riemannian gradient of a function $\bar{f} \colon \mathbb{R}^n \to \mathbb{R}$ coincides with the standard Euclidean gradient $\nabla \bar{f}$, i.e.,~\mbox{$\nabla \bar{f}(x) := (\partial \bar{f}(x)/\partial x_i) \in T_x \mathbb{R}^n \simeq \mathbb{R}^n$} for $x \in \mathbb{R}^n$.

\subsection{Derivatives of $p$-norm functions}
Here, we investigate the derivative or Euclidean gradient of the $p$-norm-related functions in $\mathbb{R}^n$.
First, although the $p$-norm is defined for any $p \in [1, \infty]$, it is of class $C^1$ only for $p \in (1, \infty)$.
In the remainder of this section, we assume~$p \in (1, \infty)$.
Then, it is easy to verify that
\begin{equation}
    \frac{d \lvert w\rvert^p}{dw} = p \sgn(w) \lvert w \rvert^{p-1}
\end{equation}
for $w \in \mathbb{R}$.
Regarding the $p$-norm of $x \in \mathbb{R}^n$, because $\|x\|_p^p = \bm{1}^T \lvert x \rvert^p$, its partial derivative with respect to the variable $x_i$ for $i \in \{1, 2, \dots, n\}$ is
\begin{equation}
    \frac{\partial \|x\|_p^p}{\partial x_i} = \frac{\partial \lvert x_i\rvert^p}{\partial x_i} = p\sgn(x_i)\lvert x_i\rvert^{p-1}.
\end{equation}
Therefore, the gradient of the function $x \to \|x\|_p^p$ is equal to
\begin{equation}
\label{eq:grad_normpp}
    \nabla (x\mapsto \|x\|_p^p)(x) = (p \sgn(x_i) \lvert x_i \rvert^{p-1}) = p \sgn(x) \odot \lvert x\rvert^{p-1}.
\end{equation}

In the subsequent sections, we exploit the fact that the conditions $\|x\|_p = 1$ and~\mbox{$\|x\|_p^p = 1$}---both of which characterize the unit sphere $S^{n-1}_p$---are equivalent to each other.
Furthermore,~$\|x\|_p^p$ usually seems to be easier to handle than $\|x\|_p$.
For example, the gradient of the~$p$-norm function is computed as
\begin{align}
    \nabla (x \mapsto \|x\|_p)(x)
    &=\nabla \Big(x\mapsto \big(\|x\|_p^p\big)^{\frac{1}{p}}\Big)(x)\nonumber\\
    &=\frac{1}{p}(\|x\|_p^{p})^{\frac{1}{p}-1} \cdot p \sgn(x) \odot \lvert x\rvert^{p-1}\nonumber\\
    &=\frac{\sgn(x) \odot \lvert x\rvert^{p-1}}{\|x\|_p^{p-1}}.\label{eq:grad_norm}
\end{align}
We prefer to use~\eqref{eq:grad_normpp}, which provides a simpler expression, rather than~\eqref{eq:grad_norm}, unless~\eqref{eq:grad_norm} is essential in the discussion.

Note that $h(x) := \|x\|_p^p$ is not necessarily a $C^\infty$ function in $\mathbb{R}^n$.
For example, consider the case $p = 3$ and $n = 2$, where~\mbox{$h(x) = \lvert x_1\rvert^3 + \lvert x_2\rvert^3$}.
Then, we have~\mbox{$\nabla h(x) = 3\begin{pmatrix} \lvert x_1\rvert x_1 \\ \lvert x_2 \rvert x_2
\end{pmatrix}$} and~$\nabla^2 h(x) = 6\begin{pmatrix}\lvert x_1 \rvert & 0 \\ 0 & \lvert x_2 \rvert\end{pmatrix}$.
Hence, $h$ is of class~$C^2$ in $\mathbb{R}^2$.
However, since~\mbox{$\partial^2 h(x)/\partial x_1^2 = 6\lvert x_1\rvert$} (resp.~\mbox{$\partial^2 h(x)/\partial x_2^2 = 6\lvert x_2\rvert$}) is not partially differentiable with respect to $x_1$ (resp. $x_2$) at any~$(0, x_2)^T \in \mathbb{R}^2$ (resp. $(x_1, 0)^T$), $h$ is not of class $C^3$ in $\mathbb{R}^2$.
This causes nonsmoothness of $S^{2}_3$, which includes the points $(\pm 1, 0)^T$ and $(0, \pm 1)^T$, as a submanifold of $\mathbb{R}^2$.
In the next section, we will prove that $S^{n-1}_p$ with $p \in (1, \infty)$ is still at least a $C^1$ submanifold of $\mathbb{R}^n$ (Theorem~\ref{thm:submanifold}).

\section{Geometry of $S^{n-1}_p$ and tools for Riemannian optimization}
\label{sec:geometry}
In this section, we discuss the geometry of the unit sphere with $p$-norm, i.e.,
\begin{equation}
    S^{n-1}_p = \{x \in \mathbb{R}^n \mid \|x\|_p = 1\},
\end{equation}
where $1 < p < \infty$.
We use the following equivalent conditions interchangeably:
\begin{equation}
\| x\|_p = 1 \iff \| x\|_p^p = 1 \iff \bm{1}^T \lvert x\rvert^p = 1 \iff x^T (\sgn(x) \odot \lvert x\rvert^{p-1}) = 1.
\end{equation}

As expected, many properties of the Euclidean sphere $S^{n-1}_2$ analogically hold for $S^{n-1}_p$ with any $p \in (1,\infty)$, especially even integer $p$, while some do not hold for $S^{n-1}_1$ or $S^{n-1}_\infty$.

\subsection{$S^{n-1}_p$ as a Riemannian submanifold of $\mathbb{R}^n$}
First, we prove that $S^{n-1}_p$ is an embedded submanifold of $\mathbb{R}^n$.

\begin{theorem}
\label{thm:submanifold}
For $p \in (1, \infty)$, the unit sphere $S^{n-1}_p$ with $p$-norm is an (n-1)-dimensional $C^r$ embedded submanifold of $\mathbb{R}^n$, where $r = \infty$ if $p$ is an even integer,~\mbox{$r = p - 1$} if $p$ is an odd integer, and $r = \lfloor p \rfloor$, which is the largest integer less than $p$, if $p$ is not an integer.\footnote{The statement can be rewritten as follows: for any positive integer $k$, $S^{n-1}_p$ is a $C^{2k-1}$ submanifold of $\mathbb{R}^n$ if~\mbox{$2k-1 < n < 2k$}, $C^\infty$ submanifold if $n = 2k$, and $C^{2k}$ submanifold if~\mbox{$2k < n \le 2k+1$}.}
\end{theorem}
\begin{proof}
We define $h \colon \mathbb{R}^n \to \mathbb{R}$ as $h(x) := \|x\|_p^p$.
We can observe that $h$ is a $C^r$ function in $\mathbb{R}^n$, where $r$ is the integer in the statement of the theorem, as follows:
If $p$ is an even integer,~\mbox{$h(x) = \sum_{i=1}^n \lvert x_i\rvert^p = \sum_{i=1}^n x_i^p$} is clearly a $C^\infty$ function.
If $p$ is an odd integer,~\mbox{$h(x) = \sum_{i=1}^n \lvert x_i\rvert^p$} is of class $C^{p-1}$ because $\partial^{p-1} h(x) / \partial x_i^{p-1} = (p!) \lvert x_i\rvert$ is continuous for any~$i \in \{1, 2, \dots, n\}$.
Similarly, if $p$ is not an integer, $h$ is of class $C^{\lfloor p \rfloor}$ because we have
\begin{align}
\frac{\partial^{\lfloor p \rfloor} h}{\partial x_i^{\lfloor p \rfloor}}(x) =& p(p-1) \cdots (p-\lfloor p \rfloor +1) \sgn(x)^{\lfloor p \rfloor} \lvert x_i\rvert^{p - \lfloor p \rfloor}\nonumber\\
=& \frac{\Gamma(p+1)}{\Gamma(p+1-\lfloor p \rfloor)}\sgn(x)^{\lfloor p \rfloor} \lvert x_i \rvert^{p-\lfloor p \rfloor},
\end{align}
which is continuous because $p-\lfloor p \rfloor > 0$ in this case, where $\Gamma(\cdot)$ is the gamma function.
Therefore, $h$ is of class $C^r$ in every case.

Using the formula~\eqref{eq:grad_normpp}, the Jacobian matrix of $h$ at $x \in \mathbb{R}^n - \{0\}$, which is defined as~$(Jh)_x := (\partial h(x) / \partial x_i)^T \in \mathbb{R}^{1 \times n}$, is computed as
\begin{equation}
\label{eq:jacobi}
    (Jh)_x = \nabla h(x)^T = p (\sgn(x) \odot \lvert x\rvert^{p-1})^T.
\end{equation}
For any $x \in \mathbb{R}^n$ satisfying $h(x) = 1$, we have $(Jh)_x \neq 0$ because such $x$ is not $0$.
This implies that $1$ is a regular value of $h$.
Therefore, it follows from the regular level set theorem~\cite[Theorem 9.9]{tu2010introduction} that~$h^{-1}(\{1\}) = S^{n-1}_p$ is a $C^r$ embedded submanifold of $\mathbb{R}^n$, whose dimension is~\mbox{$n - \dim \mathbb{R} = n - 1$}.
This completes the proof.
\end{proof}

\begin{remark}
Note that the integer $r$ in Theorem~\ref{thm:submanifold} is not less than $1$ in every case.
Therefore,~$S^{n-1}_p$ with $p \in (1, \infty)$ is always a $C^1$ submanifold of $\mathbb{R}^n$.
In contrast, if~\mbox{$p = 1$} or $p = \infty$, the unit sphere $S^{n-1}_p$ is not a $C^1$ embedded submanifold of $\mathbb{R}^n$ because of their corners.
Indeed, the above proof fails if $p = 1$ or $p = \infty$ because~\mbox{$x \mapsto \|x\|_p$} is not a $C^1$ function in such cases.
\end{remark}
In what follows, we assume $p \in (1, \infty)$ and define smoothness regarding~$S^{n-1}_p$ as $C^r$ with~$r \ge 1$ in Theorem~\ref{thm:submanifold}.
For example, we say that a function~$f$ on $S^{n-1}_p$ is smooth if $f$ is of class $C^r$.

We endow the sphere $S^{n-1}_p$ with the Riemannian metric as
\begin{equation}
    \langle \xi, \eta\rangle_x := \xi^T \eta, \qquad \xi,\, \eta \in T_x S^{n-1}_p, \quad x \in S^{n-1}_p,
\end{equation}
which is induced from the Riemannian metric (the standard inner product)
\begin{equation}
    \langle a, b\rangle_x := a^T b, \qquad a,\, b \in T_x \mathbb{R}^n \simeq \mathbb{R}^n, \quad x \in \mathbb{R}^n
\end{equation}
in the ambient space $\mathbb{R}^n$.
Thus, $S^{n-1}_p$ is a Riemannian submanifold of $\mathbb{R}^n$.

\subsection{Tangent space, normal space, and orthogonal projection}
Defining $h(x) := \|x\|_p^p$, the tangent space $T_x S^{n-1}_p$ of $S^{n-1}_p = h^{-1}(\{1\})$ at $x$ is equal to the kernel of the linear map $\D h(x) \colon \mathbb{R}^n \simeq T_x \mathbb{R}^n \to T_{h(x)}\mathbb{R} \simeq \mathbb{R}$, i.e.,~$(\D h(x))^{-1}(\{0\})$.
Here, it follows from~\eqref{eq:jacobi} that the derivative $\D h(x)$ acts on $y \in \mathbb{R}^n$ as
\begin{equation}
    \D h(x)[y] = (Jh)_x(y) = p (\sgn(x) \odot \lvert x\rvert^{p-1})^T y.
\end{equation}
Therefore, we have
\begin{equation}
\label{eq:tangent}
    T_x S^{n-1}_p = (\D h(x))^{-1}(\{0\}) = \{\xi \in \mathbb{R}^n \mid \xi^T (\sgn(x) \odot \lvert x\rvert^{p-1}) = 0\}.
\end{equation}
Since $S^{n-1}_p$ is a Riemannian submanifold of $\mathbb{R}^n$, we can define the normal space~$N_x S^{n-1}_p$ of~$S^{n-1}_p$ at a point $x$ as the orthogonal complement of~\mbox{$T_x S^{n-1}_p \subset T_x \mathbb{R}^n \simeq \mathbb{R}^n$} in $\mathbb{R}^n$ with respect to the Riemannian metric in $\mathbb{R}^n$, i.e., the standard inner product.
From the expression~\eqref{eq:tangent}, we can observe that~$T_x S^{n-1}_p$ is a hyperplane orthogonal to the vector $\sgn(x) \odot \lvert x\rvert^{p-1} \in \mathbb{R}^n$.
Hence, we have
\begin{equation}
    N_x S^{n-1}_p := (T_x S^{n-1}_p)^\perp %\nonumber\\
    %&= \Span\{\sgn(x) \odot \lvert x \rvert^{p-1}\}\nonumber\\
    %&
    = \{\alpha \sgn(x) \odot \lvert x\rvert^{p-1} \mid \alpha \in \mathbb{R}\}.\label{eq:normal}
\end{equation}

For minimizing a smooth function $f \colon S^{n-1}_p \to \mathbb{R}$ on $S^{n-1}_p$, the Riemannian gradient of $f$ is important.
Here, the Riemannian gradient $\grad f(x)$ of $f$ at~\mbox{$x \in S^{n-1}_p$} can be obtained by orthogonally projecting $\nabla \bar{f}(x) \in \mathbb{R}^n$ onto the tangent space $T_x S^{n-1}_p$ at $x$, where $\bar{f}$ is a smooth extension of $f$ to the ambient space $\mathbb{R}^n$ and $\nabla \bar{f}(x) := (\partial \bar{f}(x) / \partial x_i) \in \mathbb{R}^n$ is the Euclidean gradient.
That is, we have
\begin{equation}
    \grad f(x) = P_x(\nabla \bar{f}(x)),
\end{equation}
where $P_x$ is the orthogonal projection to the tangent space $T_x S^{n-1}_p$ at $x$.
The projection $P_x \colon \mathbb{R}^n \to T_x S^{n-1}_p$ acts on any $d \in \mathbb{R}^n$ so that $d - P_x(d) \in N_ x S^{n-1}_p$ holds.
From~\eqref{eq:normal}, the normal vector~$d - P_x(d)$ is written as $\alpha \sgn(x) \odot \lvert x\rvert^{p-1}$ for some $\alpha \in \mathbb{R}$.
Thus, we obtain the decomposition of $d$ as
\begin{equation}
\label{eq:d}
    d = P_x(d) + \alpha \sgn(x) \odot \lvert x\rvert^{p-1}.
\end{equation}
By noting the expression~\eqref{eq:tangent} and multiplying~\eqref{eq:d} by $(\sgn(x) \odot \lvert x\rvert^{p-1})^T$ from the left, we obtain $\alpha = ((\sgn(x) \odot \lvert x\rvert^{p-1})^T d) / \|\lvert x\rvert^{p-1}\|_2^2$, where we used the relation
\begin{equation}
    (\sgn(x) \odot \lvert x\rvert^{p-1})^T(\sgn(x) \odot \lvert x\rvert^{p-1})
    = ((\sgn(x))^2)^T(\lvert x\rvert^{p-1})^2
%    &= \bm{1}^T (x^{p-1})^2\\
    = \| \lvert x\rvert^{p-1}\|_2^2 \neq 0.
\end{equation}
Substituting the expression of $\alpha$ to~\eqref{eq:d}, we obtain
\begin{align}
    P_x(d) &= d - \frac{(\sgn(x) \odot \lvert x\rvert^{p-1})^T d}{\|\lvert x\rvert^{p-1}\|_2^2}\sgn(x) \odot \lvert x\rvert^{p-1}\nonumber\\
    &= \left(I - \frac{(\sgn(x) \odot \lvert x\rvert^{p-1})(\sgn(x) \odot \lvert x\rvert^{p-1})^T}{\|\lvert x\rvert^{p-1}\|_2^2}\right)d.
\end{align}
In other words, the linear map $P_x$ is represented as the matrix
\begin{equation}
\label{eq:proj}
    P_x = I - \frac{(\sgn(x) \odot \lvert x\rvert^{p-1})(\sgn(x) \odot \lvert x\rvert^{p-1})^T}{\|\lvert x\rvert^{p-1}\|_2^2}.
\end{equation}

\section{Retractions and their inverses}
\label{sec:retraction}
In an iterative Riemannian optimization algorithm on a Riemannian manifold~$\mathcal{M}$, to compute the next point from the current point~\mbox{$x \in \mathcal{M}$} and search direction $\eta \in T_x \mathcal{M}$, a retraction on $\mathcal{M}$ is important~\mbox{\cite{AbsMahSep2008,adler2002newton,shub1986some}}.
A map~\mbox{$R \colon T\mathcal{M} \to \mathcal{M}$} is said to be a retraction on $\mathcal{M}$ if the restriction~\mbox{$R_x := R\vert_{T_x \mathcal{M}}$} of $R$ to $T_x \mathcal{M}$ for~$x \in \mathcal{M}$ satisfies $R_x(0_x) = x$ and~\mbox{$\D R_x(0_x) = \id_{T_x \mathcal{M}}$}, where $0_x$ is the zero vector in $T_x \mathcal{M}$ and $\id_{T_x \mathcal{M}}$ is the identity map in $T_x \mathcal{M}$.
Although retractions are usually discussed on $C^\infty$ manifolds, the manifold $S^{n-1}_p$ is a $C^r$ submanifold of $\mathbb{R}^n$, where~$r$ is in Theorem~\ref{thm:submanifold} and may not be $\infty$.
Therefore, we define a retraction on $S^{n-1}_p$ as a $C^r$, which we say smooth, map on $S^{n-1}_p$ satisfying the above properties.

Furthermore, the inverse of a retraction can be used in, e.g., the Riemannian conjugate gradient method~\cite{zhu2020riemannian}.
In the following, we discuss three types of retractions on $S^{n-1}_p$ and their respective inverses.

\subsection{Retraction by normalization and its inverse}
\label{subsec:retr_normalization}
Intuitively, for any $x \in S^{n-1}_p$ and $\eta \in T_x S^{n-1}_p$, $x + \eta \in \mathbb{R}^n$ appears to be outside $S^{n-1}_p$ unless~$\eta = 0$.
This is actually true from the following proposition.
However, its proof for general $p > 1$ is not as easy as in the case of $p = 2$.

\begin{proposition}
\label{prop:norm}
Assume that $p \in (1, \infty)$.
For any $x \in S^{n-1}_p$ and $\eta \in T_x S^{n-1}_p$, if~\mbox{$\eta \neq 0$}, then~$\|x + \eta\|_p > 1$ holds.
\end{proposition}

\begin{proof}
Note that the function $h(y) := \|y\|_p^p$ is not of class $C^2$ in the entire $\mathbb{R}^n$ when~\mbox{$1 < p < 2$}.
Therefore, we avoid using the Hessian matrix in the following discussion to address the general case.

We first show that $h$ is a strictly convex function in $\mathbb{R}^n$.
For $y, z \in \mathbb{R}^n$ with $y \neq z$ and~$\alpha \in (0, 1)$, Minkowski's inequality (the triangle inequality for the $p$-norm) as well as convexity and monotonicity of the function $w \mapsto w^p$ on $\mathbb{R}_+ := \{w \in \mathbb{R} \mid w \ge 0\}$ yield that
\begin{equation}
\label{eq:minkowski}
    \|\alpha y + (1-\alpha) z\|_p^p \le (\alpha \|y\|_p + (1-\alpha)\|z\|_p)^p
    \le \alpha \|y\|_p^p + (1-\alpha)\|z\|_p^p.
\end{equation}
We now assume that both equalities in \eqref{eq:minkowski} simultaneously hold.
Then, the first equality implies that $y = cz$ for some $c \ge 0$ or $z = 0$ from Minkowski's inequality theory for $p \in (1, \infty)$.
Furthermore, from the second equality and the strict convexity of $w \mapsto  w^p$ on $\mathbb{R}_+$, we have~$\|y\|_p = \|z\|_p$.
If $y = cz$ with $c \ge 0$, then $\|y\|_p = \|z\|_p$ implies $c = 1$ or $\|y\|_p = \|z\|_p = 0$.
Otherwise, we have $z = 0$; and $\|y\|_p = \|z\|_p$ then means $y = z = 0$.
In any case, we have $y = z$, which contradicts the assumption that~$y \neq z$.
Therefore, both equalities in~\eqref{eq:minkowski} do not hold at the same time, meaning
\begin{equation}
        h(\alpha y + (1-\alpha) z) = \|\alpha y + (1-\alpha) z\|_p^p < \alpha \|y\|_p^p + (1-\alpha)\|z\|_p^p = \alpha h(y) + (1-\alpha)h(z).
\end{equation}
This proves that $h$ is strictly convex.

By using the strict convexity of $h$, we can show that $\phi(t) := h(x+t\eta) = \|x + t\eta\|_p^p$ is a strictly convex function on $\mathbb{R}$ for $x \in S^{n-1}_p$ and $\eta \in T_x S^{n-1}_p$ with $\eta \neq 0$.
Indeed, for any $s, t \in \mathbb{R}$ with $s \neq t$ and $\alpha \in (0,1)$, it follows from the strict convexity of $h$ and the fact $x+s\eta \neq x+t\eta$ that
\begin{align}
    \phi(\alpha s + (1-\alpha) t)
    &= h(x + (\alpha s + (1-\alpha) t)\eta)\nonumber\\
    &=h(\alpha(x + s \eta) + (1-\alpha)(x + t \eta))\nonumber\\
    &< \alpha h(x+s\eta) + (1-\alpha)h(x+t\eta)\nonumber\\
    &= \alpha\phi(s) + (1-\alpha)\phi(t).
\end{align}

Subsequently, we show that $t = 0$ is the unique minimizer of $\phi$.
Since $\phi$ is strictly convex, it suffices to prove that $\phi'(0) = 0$, which is shown as
\begin{equation}
    \phi'(0) = \nabla h(x)^T \eta = p(\sgn(x) \odot \lvert x\rvert^{p-1})^T\eta = 0
\end{equation}
from~\eqref{eq:grad_normpp} and~\eqref{eq:tangent}.

In conclusion, we obtain $\|x+t\eta\|_p^p = \phi(t) > \phi(0) = \|x\|_p^p = 1$ for all $t \neq 0$, where the case of $t = 1$ implies that the desired inequality $\|x+\eta\|_p > 1$ holds.
\end{proof}

Considering Proposition~\ref{prop:norm}, we propose a retraction $R$ on $S^{n-1}_p$ as
\begin{equation}
\label{eq:retraction}
    R_x(\eta) := \frac{x+\eta}{\|x+\eta\|_p}, \qquad \eta \in T_x S^{n-1}_p, \quad x \in S^{n-1}_p.
\end{equation}
This is simply the normalization (with respect to the $p$-norm) of~\mbox{$x + \eta$}, which is not on $S^{n-1}_p$ when $\eta \neq 0$.
Note that the denominator in~\eqref{eq:retraction} is ensured to be nonzero from Proposition~\ref{prop:norm}.

\begin{proposition}
Assume that $p \in (1, \infty)$.
The map $R$ defined by~\eqref{eq:retraction} is a retraction on~$S^{n-1}_p$.
\end{proposition}

\begin{proof}
It is clear that $\|R_x(\eta)\|_p = 1$ and $R_x(0_x) = x$ hold for any $x \in S^{n-1}_p$ and~\mbox{$\eta \in T_x S^{n-1}_p$}.
To prove that $\D R_x(0_x) = \id_{T_x S^{n-1}_p}$ holds, we use~\eqref{eq:grad_norm}, i.e., the fact that the gradient of $x\mapsto \|x\|_p$ is written as $\|x\|_p^{1-p}\sgn(x) \odot \lvert x\rvert^{p-1}$.
Then, we can compute $\D R_x(0_x)[\eta]$ for $\eta \in T_x S^{n-1}_p$ as
\begin{align}
     & \quad \ \D R_x(0_x)[\eta] = \frac{d}{dt}R_x(t\eta)\bigg\vert_{t=0}\nonumber\\
    &= \frac{\eta\|x+t\eta\|_p - (x+t\eta)\big(\|x+t\eta\|_p^{1-p}\sgn(x+t\eta) \odot \lvert x+t\eta\rvert^{p-1}\big)^T \eta}{\|x+t\eta\|_p^2}\bigg\vert_{t=0}\nonumber\\
    &= \eta - x(\sgn(x)\odot \lvert x\rvert^{p-1})^T \eta = \eta,\label{eq:DR0}
\end{align}
where we used $\|x\|_p = 1$ and $(\sgn(x)\odot \lvert x\rvert^{p-1})^T \eta = 0$ from~\eqref{eq:tangent}.
\end{proof}

To derive the inverse of $R$, we fix $x, y \in S^{n-1}_p$ and assume that $\eta \in T_x S^{n-1}_p$ satisfies~\mbox{$R_x(\eta) = y$}.
Then, $\eta$ should satisfy $x + \eta = \alpha y$ for some $\alpha > 0$.
Multiplying the equality by $(\sgn(x) \odot \lvert x \rvert^{p-1})^T$ from the left and noting that~\mbox{$x \in S^{n-1}_p$} and $\eta \in T_x S^{n-1}_p$, we obtain~$\alpha = 1 / ((\sgn(x) \odot \lvert x \rvert^{p-1})^T y)$.
Therefore, $\eta$ should satisfy
\begin{equation}
\label{eq:eta}
    \eta = \alpha y - x = \frac{y}{(\sgn(x) \odot \lvert x \rvert^{p-1})^T y} - x.
\end{equation}
However, this is necessary but not sufficient for $R_x(\eta) = y$.
In fact, for certain $x, y \in S^{n-1}_p$, there may not exist $\eta$ such that $R_x(\eta) = y$.
The following proposition elaborates on this issue.

\begin{proposition}
Assume that $p \in (1, \infty)$.
For any $x \in S^{n-1}_p$, the inverse of $R_x$ defined in~\eqref{eq:retraction} is given by
\begin{equation}
\label{eq:invretr}
    R_x^{-1}(y) = \frac{y}{(\sgn(x) \odot \lvert x \rvert^{p-1})^T y} - x, \qquad y \in D_x
\end{equation}
where the domain $D_x$ of $R_x^{-1}$ is $D_x = \{y \in S^{n-1}_p \mid (\sgn(x) \odot \lvert x \rvert^{p-1})^T y > 0\}$.
\end{proposition}

\begin{proof}
For $y$ satisfying $(\sgn(x) \odot \lvert x \rvert^{p-1})^T y = 0$, the right-hand side of~\eqref{eq:invretr} is not defined.
We assume that $(\sgn(x) \odot \lvert x \rvert^{p-1})^T y \neq 0$ and denote the right-hand side of~\eqref{eq:invretr} by $\eta_{x,y}$, which is in $T_x S^{n-1}_p$ because $(\sgn(x) \odot \lvert x \rvert^{p-1})^T \eta_{x,y} = 1 -1 = 0$.
Then, we have
\begin{align}
    R_x(\eta_{x,y}) &= \frac{x + \eta_{x,y}}{\|x + \eta_{x,y}\|_p} = \frac{y}{\sgn((\sgn(x) \odot \lvert x \rvert^{p-1})^T y)} \nonumber\\
    &= \begin{cases}
    y \quad \text{if $(\sgn(x) \odot \lvert x \rvert^{p-1})^T y > 0$}\\
    -y \quad \text{if $(\sgn(x) \odot \lvert x \rvert^{p-1})^T y < 0$}
    \end{cases}.
\end{align}
Furthermore, if $\eta \in T_x S^{n-1}_p$ satisfies $R_x(\eta) = y$, then $\eta$ should be equal to $\eta_{x,y}$, as discussed in~\eqref{eq:eta}.
Therefore, $R_x(\eta) = y$ holds if and only if $(\sgn(x) \odot \lvert x \rvert^{p-1})^T y > 0$ and $\eta = \eta_{x,y}$.
This completes the proof.
\end{proof}

\subsection{Inverse of projective retraction}
\label{subsec:retr_proj}
Another natural retraction is the projective retraction~\cite{absil2012projection}.
The projective retraction $R^{\proj}$ on~$S^{n-1}_p$ is given by
\begin{equation}
\label{eq:Rproj}
    R^{\proj}_x(\eta) = \argmin_{y \in S^{n-1}_p}\|(x+\eta)-y\|_2, \qquad \eta \in T_x S^{n-1}_p, \quad x \in S^{n-1}_p.
\end{equation}

\begin{remark}
\label{rem:proj_retr}
Note that the projection onto any closed convex set in $\mathbb{R}^n$ regarding the~\mbox{$2$-norm} is unique~\cite[Section 8.1]{boyd2004convex}.
Therefore, because the unit ball~\mbox{$B^{n}_p := \{x \in \mathbb{R}^n \mid \|x\|_p \le 1\}$} with~$p$-norm is obviously a closed convex set in $\mathbb{R}^n$,
vector $y \in B^{n}_p$ that minimizes the distance $\|(x+\eta) - y\|_2$ uniquely exists for a given $x \in S^{n-1}_p$ and $\eta \in T_x S^{n-1}_p$.
Since~$x+\eta$ is outside~$B^{n-1}_p$ unless $\eta = 0$ from Proposition~\ref{prop:norm}, the right-hand side in~\eqref{eq:Rproj} is equal to the uniquely existing projection of~$x+\eta$ onto $B^{n}_p$ (clearly, we have~$R^{\proj}_x(\eta) = x$ when $\eta = 0$).
\end{remark}

The vector $R^{\proj}_x(\eta)$ satisfies $(x+\eta) - R^{\proj}_x(\eta) \in N_x S^{n-1}_p$, which is implied by~\cite{absil2012projection} or is a direct consequence of the Lagrange multiplier method.
Therefore, there exists $\alpha \in \mathbb{R}$ such that
\begin{equation}
\label{eq:proj1}
    R^{\proj}_x(\eta) = x+\eta - \alpha\sgn(R^{\proj}_x(\eta)) \odot \lvert R^{\proj}_x(\eta) \rvert^{p-1},
\end{equation}
where $\alpha$ is determined such that $R^{\proj}_x(\eta) \in S^{n-1}_p$ holds, i.e.,
\begin{equation}
\label{eq:proj2}
    \|x+\eta - \alpha\sgn(R^{\proj}_x(\eta)) \odot \lvert R^{\proj}_x(\eta) \rvert^{p-1}\|_p = 1.
\end{equation}
However, it may be difficult to explicitly express $R^{\proj}_x(\eta)$ by solving~\eqref{eq:proj1} and~\eqref{eq:proj2}.

\begin{remark}
\label{rem:p2}
When $p = 2$, Eq.~\eqref{eq:proj1} is reduced to $R^{\proj}_x(\eta) = x+\eta - \alpha R^{\proj}_x(\eta)$, i.e., we have~$(\alpha + 1)R^{\proj}_x(\eta) = (x+\eta)$.
Then, $\|R^{\proj}_x(\eta)\|_2 = 1$ implies $\lvert \alpha + 1\rvert = \|x + \eta\|_2$.
Hence, we obtain $(x+\eta) / (\alpha + 1) = \pm (x+\eta) / \|x+\eta\|_2$, among which $(x+\eta) / \|x+\eta\|_2$ is closer to~$x+\eta$.
In summary, when $p = 2$, we have $R^{\proj}_x(\eta) = (x+\eta) / \|x+\eta\|_2$, which is equal to the retraction by normalization in Section~\ref{subsec:retr_normalization}.
\end{remark}

Although the above discussion implies that the projective retraction on~$S^{n-1}_p$ for general~\mbox{$p \in (1, \infty)$} may not provide as successful a result as the retraction by normalization, the inverse of $R^{\proj}_x$ can be discussed more practically.
For given~\mbox{$x, y \in S^{n-1}_p$}, if $\eta \in T_x S^{n-1}_p$ satisfies~$R^{\proj}_x(\eta) = y$, then~$x + \eta - y \in N_y S^{n-1}_p$ should hold.
Therefore, there exists $\alpha_{x,y} \in \mathbb{R}$ such that $x + \eta - y = \alpha_{x,y} \sgn(y) \odot \lvert y\rvert^{p-1}$.
From $x \in S^{n-1}_p$ and $\eta \in T_x S^{n-1}_p$, we can obtain an explicit expression for $\alpha_{x,y}$ as in Proposition~\ref{prop:inv_proj_retr}.
Furthermore, it seems that $\alpha_{x,y}$ should be nonnegative by analogy with the discussion of the case $p=2$ in Remark~\ref{rem:p2}.
We discuss these rigorously in the proof of the proposition using the Karush--Kuhn--Tucker (KKT) conditions.

\begin{proposition}
\label{prop:inv_proj_retr}
Assume that $p \in (1, \infty)$.
For any $x \in S^{n-1}_p$, the inverse of $R^{\proj}_x$ in~\eqref{eq:Rproj} is given by
\begin{align}
\label{eq:retr_proj}
    (R^{\proj}_x)^{-1}(y) &= y-x + \alpha_{x,y}\sgn(y) \odot \lvert y\rvert^{p-1}\nonumber\\
    &= \left(I - \frac{(\sgn(y) \odot \lvert y\rvert^{p-1})(\sgn(x) \odot \lvert x\rvert^{p-1})^T}{(\sgn(y) \odot \lvert y\rvert^{p-1})^T(\sgn(x) \odot \lvert x\rvert^{p-1})}\right)(y-x), \qquad y \in D_x,
\end{align}
where
\begin{equation}
\label{eq:alpha_proj2}
    \alpha_{x,y} := \frac{1 - (\sgn(x) \odot \lvert x\rvert^{p-1})^T y}{(\sgn(x) \odot \lvert x\rvert^{p-1})^T(\sgn(y) \odot \lvert y\rvert^{p-1})}
\end{equation}
and the domain of $(R^{\proj}_x)^{-1}$ is 
\begin{equation}
\label{eq:domain_proj}
    D_x = \{y \in S^{n-1}_p \mid (\sgn(x) \odot \lvert x\rvert^{p-1})^T(\sgn(y) \odot \lvert y\rvert^{p-1}) \neq 0,\, \alpha_{x,y} \ge 0\}.
\end{equation}
\end{proposition}

\begin{proof}
The second equality in~\eqref{eq:retr_proj} directly follows from~\mbox{$(\sgn(x) \odot \lvert x \rvert^{p-1})^T x = 1$}.
We define~$\eta_{x,y} := y - x + \alpha_{x,y} \sgn(y) \odot \lvert y \rvert^{p-1}$ with $\alpha_{x,y}$ in~\eqref{eq:alpha_proj2}.
Then, what we need to prove is that~\mbox{$R^{\proj}_x(\eta) = y$} holds for $\eta \in T_x S^{n-1}_p$ if and only if $y$ belongs to the right-hand side of~\eqref{eq:domain_proj} and $\eta = \eta_{x,y}$.

To see this in light of~\eqref{eq:Rproj} and Remark~\ref{rem:proj_retr}, we must verify that $z = y$ is the optimal solution to the following optimization problem with a fixed $\eta \in T_x S^{n-1}_p$ if and only if $\alpha_{x,y}$ in~\eqref{eq:alpha_proj2} is well-defined and nonnegative and $\eta = \eta_{x,y}$:
\begin{alignat*}{2}
&\text{minimize}&\quad & \|(x+\eta) - z\|_2^2\\
&\text{subject to}&\quad & \|z\|_p^p \le 1, \ z \in \mathbb{R}^n,
\end{alignat*}
where the decision variable vector is $z$.
This is a convex optimization problem because both~$z \mapsto \|(x+\eta) - z\|_2^2$ and $z \mapsto \|z\|_p^p - 1$ are convex.
Furthermore, the problem satisfies Slater's condition~\cite[Section 5.2.3]{boyd2004convex}, i.e., it is strictly feasible (e.g., with $z = 0$).
Therefore, the condition that~$z = y$ is optimal for the optimization problem is equivalent to saying that there exists~$\lambda \in \mathbb{R}$ such that $z = y$ and $\lambda$ satisfy the KKT conditions for the problem, which are written as
\begin{align}
    2(y-(x+\eta)) + \lambda p \sgn(y) \odot \lvert y\rvert^{p-1} = 0,\\
    \| y \|_p^p \le 1,\\
    \lambda \ge 0,\\
    \lambda(\|y\|_p^p -1) = 0.
\end{align}
Since $\|y\|_p = 1$, they are equivalent to
\begin{align}
    \eta = y-x + \frac{p}{2}\lambda\sgn(y) \odot \lvert y\rvert^{p-1},\label{eq:etaxy}\\
    \lambda \ge 0.\label{eq:etaxy2}
\end{align}
Noting that $x \in S^{n-1}_p$ and $\eta \in T_x S^{n-1}_p$, we multiply~\eqref{eq:etaxy} by $(\sgn(x) \odot \lvert x\rvert^{p-1})^T$ from the left to obtain
\begin{equation}
    2(1 - (\sgn(x) \odot \lvert x\rvert^{p-1})^T y) = p\lambda (\sgn(x) \odot \lvert x\rvert^{p-1})^T (\sgn(y) \odot \lvert y\rvert^{p-1}).
\end{equation}
If $(\sgn(x) \odot \lvert x\rvert^{p-1})^T(\sgn(y) \odot \lvert y \rvert^{p-1}) = 0$ holds, $1 - (\sgn(x) \odot \lvert x \rvert^{p-1})^T y = 0$ should hold, and~$\lambda$ can be any value. 
However, it then follows from~\eqref{eq:etaxy} that $y = 0$, contradicting $y \in S^{n-1}_p$.
Hence, we have $(\sgn(x) \odot \lvert x\rvert^{p-1})^T(\sgn(y) \odot \lvert y \rvert^{p-1}) \neq 0$, and~$\lambda$ is written as
\begin{equation}
    \lambda = \frac{2}{p}\frac{1-(\sgn(x) \odot \lvert x\rvert^{p-1})^T y}{(\sgn(x) \odot \lvert x\rvert^{p-1})^T (\sgn(y) \odot \lvert y\rvert^{p-1})} = \frac{2}{p}\alpha_{x,y}.
\end{equation}
Therefore, there exists $\lambda \in \mathbb{R}$ such that $z = y$ and $\lambda$ satisfy the KKT conditions~\eqref{eq:etaxy} and~\eqref{eq:etaxy2} if and only if $\eta = y-x + \alpha_{x,y} \sgn(y) \odot \lvert y\rvert^{p-1} = \eta_{x,y}$ and $\alpha_{x,y}$ is well-defined and nonnegative.
This completes the proof.
\end{proof}

\subsection{Inverse of orthographic retraction}
Other possibilities of retractions on $S^{n-1}_p$ include the orthographic retraction.
See~\cite{absil2012projection} for a discussion of orthographic retractions on general Riemannian submanifolds.

For $x \in S^{n-1}_p$ and $\eta \in T_x S^{n-1}_p$, the orthographic retraction $R^{\orth}$ is defined to satisfy~$R^{\orth}_x(\eta) = x + \eta + \zeta \in S^{n-1}_p$ for some $\zeta \in N_x S^{n-1}_p$ with the smallest norm among all normal vectors in $\{\xi \in N_x S^{n-1}_p \mid x+\eta+\xi\ \in S^{n-1}_p\}$.
Since we can express $\zeta \in N_x S^{n-1}_p$ as~$\zeta = -\alpha \sgn(x) \odot \lvert x\rvert^{p-1}$ for some $\alpha \in \mathbb{R}$, the relation $\|R^{\orth}_x(\eta)\|_p = 1$ yields the equation on~$\alpha$ as
\begin{equation}
\label{eq:alpha}
    \|x+\eta-\alpha\sgn(x) \odot \lvert x\rvert^{p-1}\|_p^p = 1.
\end{equation}

\begin{remark}
When $p = 2$, Eq.~\eqref{eq:alpha} is reduced to $(1-\alpha)^2 + \eta^T \eta = 1$, the smaller solution (with smaller absolute value) of which is given by $\alpha = 1 - \sqrt{1-\eta^T\eta}$ if $\|\eta\|_2 \le 1$.
This gives the expression $R^{\orth}_x(\eta) = \sqrt{1-\eta^T\eta} \, x + \eta$, which is a well-known result.
\end{remark}

For general $p \in (1, \infty)$, we have
\begin{equation}
    R^{\orth}_x(\eta) = x + \eta - \alpha\sgn(x) \odot \lvert x\rvert^{p-1}, \qquad \eta \in T_x S^{n-1}_p, \quad x \in S^{n-1}_p,
\end{equation}
where $\eta$ should be a tangent vector such that Eq.~\eqref{eq:alpha} has a solution and $\alpha$ is the one with the smallest absolute value of the solutions.
Unfortunately, as in the projective retraction in Section~\ref{subsec:retr_proj}, it may be difficult to explicitly express such $\alpha$ for general $p$.

However, the discussion on this retraction is still important because its inverse can be practically computed.
Subsequently, we assume that $\eta \in T_x S^{n-1}_p$ satisfies $R^{\orth}_x(\eta) = y$ for given $x, y \in S^{n-1}_p$.
Then, there exists $\alpha_{x,y} \in \mathbb{R}$ such that $x + \eta - \alpha_{x,y} \sgn(x) \odot \lvert x\rvert^{p-1} = y$.
Since $\eta \in T_x S^{n-1}_p$, multiplying both sides by $(\sgn(x) \odot \lvert x\rvert^{p-1})^T$ from the left yields
\begin{equation}
\label{eq:alpha_orth}
    \alpha_{x,y} = \frac{1-(\sgn(x) \odot \lvert x\rvert^{p-1})^T y}{(\sgn(x) \odot \lvert x\rvert^{p-1})^T(\sgn(x) \odot \lvert x\rvert^{p-1})} = \frac{1-(\sgn(x) \odot \lvert x\rvert^{p-1})^T y}{\|\lvert x\rvert^{p-1}\|_2^2}.
\end{equation}
Note that the denominator is nonzero because of $x \neq 0$.
This observation, together with the discussion on when~\eqref{eq:alpha_orth} is sufficient for $R^{\orth}_x(\eta) = y$, leads to the following proposition.

\begin{proposition}
Assume that $p \in (1, \infty)$.
For any $x \in S^{n-1}_p$, the inverse of the retraction~$R^{\orth}_x$ is given by
\begin{align}
\label{eq:retr_orth}
    (R^{\orth}_x)^{-1}(y) &= y-x + \alpha_{x,y}\sgn(x) \odot \lvert x\rvert^{p-1}, \qquad y \in D_x,
\end{align}
where
\begin{equation}
\label{eq:alpha_orth_prop}
    \alpha_{x,y} := \frac{1-(\sgn(x) \odot \lvert x\rvert^{p-1})^T y}{\|\lvert x\rvert^{p-1}\|_2^2}
\end{equation}
and the domain of $(R^{\orth}_x)^{-1}$ is 
\begin{equation}
\label{eq:domain_orth}
    D_x = \{y \in S^{n-1}_p \mid \text{$\alpha_{x,y}$ is the solution to~\eqref{eq:alpha} with the smallest absolute value}\}.
\end{equation}
\end{proposition}

\begin{proof}
Let $\eta_{x,y} := y - x + \alpha_{x,y} \sgn(x) \odot \lvert x\rvert^{p-1}$ be the right-hand side of~\eqref{eq:retr_orth} with~$\alpha_{x,y}$ in~\eqref{eq:alpha_orth_prop}.
From the above discussion on~\eqref{eq:alpha_orth}, for a given $x \in S^{n-1}_p$, if $\eta \in T_x S^{n-1}_p$ satisfies~$R^{\orth}_x(\eta) = y$, then $y$ should belong to the right-hand side of~\eqref{eq:domain_orth} and~\mbox{$\eta = \eta_{x,y}$} should hold.

To prove the converse, we show that $R^{\orth}_x(\eta) = y$ holds if $y$ belongs to the right-hand side of~\eqref{eq:domain_orth} and $\eta=\eta_{x,y}$ holds.
Assume that $y$ and $\eta$ be such vectors, i.e.,~\mbox{$\alpha = \alpha_{x,y}$} is the solution to~\eqref{eq:alpha} with the smallest absolute value and $\eta = \eta_{x,y}$.
Then, from the definition of the orthographic retraction and the expression of~$\eta_{x,y}$, we have
\begin{equation}
    R^{\orth}_x(\eta) = R^{\orth}_x(\eta_{x,y}) = x + \eta_{x,y} - \alpha_{x,y} \sgn(x) \odot \lvert x\rvert^{p-1} = y.
\end{equation}
This completes the proof.
\end{proof}

\subsection{Discussion on exponential retraction}
On a general Riemannian manifold, another important retraction is the exponential retraction~$R := \Exp$, where~$\Exp$ is the exponential map.
However, it may be difficult to use practically.
Here, we discuss this issue.
In the following discussion, we assume that $p \ge 2$, which ensures that $S^{n-1}_p$ is a $C^2$ submanifold of $\mathbb{R}^n$ from Theorem~\ref{thm:submanifold}. 

The exponential map $\Exp$ is defined as
\begin{equation}
    \Exp_x(\eta) := \gamma_{x, \eta}(1), \qquad \eta \in T_x S^{n-1}_p, \quad x \in S^{n-1}_p,
\end{equation}
where $\gamma_{x, \eta}$ is the geodesic on $S^{n-1}_p$ emanating from $x$ in the direction of~$\eta$.
The geodesic satisfies the geodesic equation, which is derived from the condition~$\ddot \gamma_{x, \eta}(t) \in N_{\gamma_{x,\eta}(t)} S^{n-1}_p$.
For simplicity, we denote $\gamma_{x, \eta}(t)$ by $x(t)$.
Then,~$x(t) \in S^{n-1}_p$ implies $\bm{1}^T \lvert x(t)\rvert^p = 1$.
Differentiating both sides, we obtain~$(\sgn(x(t)) \odot \lvert x(t)\rvert^{p-1})^T \dot x(t) = 0$.
We further differentiate both sides to get
\begin{equation}
\label{eq:ddot}
(p-1) (\lvert x(t)\rvert^{p-2} \odot \dot x(t))^T\dot x(t) + (\sgn(x(t)) \odot \lvert x(t)\rvert^{p-1})^T\ddot x(t) = 0.
\end{equation}
From $\ddot x(t) \in N_{x(t)}S^{n-1}_p$, there exists $\alpha(t) \in \mathbb{R}$ such that~\mbox{$\ddot x(t) = \alpha(t) \sgn(x(t)) \odot \lvert x(t)\rvert^{p-1}$}.
Substituting this into~\eqref{eq:ddot}, we obtain~$\alpha(t) = -(p-1)((\lvert x(t)\rvert^{p-2})^T\dot x(t)^2) / \|\lvert x(t)\rvert^{p-1}\|_2^2$.
Therefore,~$x(t)$ satisfies the geodesic equation
\begin{equation}
\label{eq:geodesic}
    \ddot x(t) + \frac{(p-1)(\lvert 
    x(t)\rvert^{p-2})^T\dot x(t)^2}{\|\lvert x(t)\rvert^{p-1}\|_2^2}\sgn(x(t)) \odot \lvert x(t)\rvert^{p-1} = 0.
\end{equation}
Solving this equation for the case $p \neq 2$ may be difficult.
Thus, this will be dealt in a future work.

\begin{remark}
When $p = 2$, Eq.~\eqref{eq:geodesic} is reduced to $\ddot x(t) + (\dot x(t)^T \dot x(t)) x(t) = 0$, whose solution is~$x(t) = x \cos (\|\eta\|_2 t) + (\eta / \|\eta\|_2)\sin (\|\eta\|_2 t)$, where $x(0) = x$ and $\dot x(0) = \eta$, as shown in~\cite[Example 5.4.1]{AbsMahSep2008}.
\end{remark}

\section{Vector transports}
\label{sec:VT}
In addition to a retraction, a vector transport is also an important geometric tool in Riemannian optimization methods, e.g., Riemannian conjugate gradient methods~\cite{AbsMahSep2008,ring2012optimization,sakai2021sufficient,sato2016dai,sato2015new}, Riemannian quasi-Newton methods~\cite{huang2018riemannian,huang2015broyden}, and Riemannian stochastic optimization methods~\cite{sato2019riemannian,zhou2019faster}.
Let~$\mathcal{M}$ be a Riemannian manifold and $T\mathcal{M} \oplus T\mathcal{M} := \{(\eta, \xi) \mid \eta,\, \xi \in T_x \mathcal{M}, \ x \in \mathcal{M}\}$ be the Whitney sum.
A map $\mathcal{T} \colon T\mathcal{M} \oplus T\mathcal{M} \to \mathcal{M}$ is called a vector transport on $\mathcal{M}$ if there exists a retraction~$R$ on $\mathcal{M}$ and the following conditions are satisfied for any~\mbox{$x \in \mathcal{M}$}:
(i)~$\mathcal{T}_{\eta}(\xi) \in T_{R_x(\eta)}\mathcal{M}$ for any~$\eta, \, \xi \in T_x \mathcal{M}$;
(ii) $\mathcal{T}_{0_x} = \id_{T_x\mathcal{M}}$;
(iii)~$\mathcal{T}_{\eta}$ is a linear transformation in $T_x \mathcal{M}$ for any $\eta \in T_x \mathcal{M}$.

\subsection{Differentiated retraction}
An important vector transport is the differentiated retraction $\mathcal{T}^R$~\cite[Section 8.1.2]{AbsMahSep2008} associated with a retraction~$R$ on $S^{n-1}_p$ defined by
\begin{equation}
    \mathcal{T}^R_{\eta}(\xi) := \D R_x(\eta)[\xi], \qquad \eta,\, \xi \in T_x S^{n-1}_p, \quad x \in S^{n-1}_p.
\end{equation}
The differentiated retraction appears in the Riemannian (strong) Wolfe conditions and is thus used for line search in various algorithms.

Here, we derive the expression of $\mathcal{T}^R$ with the retraction $R$ defined in~\eqref{eq:retraction}.
Noting~\eqref{eq:grad_norm}, an analogous computation to~\eqref{eq:DR0} gives
\begin{align}
\label{eq:VT}
    \mathcal{T}^R_{\eta}(\xi)
    & = \D R_x(\eta)[\xi] % \\&
    = \frac{d}{dt}R_x(\eta + t\xi)\bigg\vert_{t=0}\nonumber\\
    &= \frac{\xi\|x+\eta\|_p - (x+\eta)\big(\|x+\eta\|_p^{1-p}\sgn(x+\eta) \odot \lvert x+\eta\rvert^{p-1}\big)^T \xi}{\|x+\eta\|_p^2}\nonumber\\
    &= \frac{\xi}{\|x+\eta\|_p} - \frac{(\sgn(x+\eta) \odot \lvert x+\eta\rvert^{p-1}\big)^T \xi}{\|x+\eta\|_p^{p+1}}(x+\eta).    
\end{align}

\subsection{Vector transport based on orthogonal projection}
Since $S^{n-1}_p$ is a Riemannian submanifold of $\mathbb{R}^n$, another vector transport $\mathcal{T}^P$ on $S^{n-1}_p$ is defined by the orthogonal projection~\cite[Section 8.1.3]{AbsMahSep2008} as
\begin{equation}
    \mathcal{T}^P_{\eta}(\xi) := P_{R_x(\eta)}(\xi), \qquad \eta,\, \xi \in T_x S^{n-1}_p, \quad x \in S^{n-1}_p,
\end{equation}
where the orthogonal projection $P$ is provided by~\eqref{eq:proj}.
Specifically, if we use the retraction~\eqref{eq:retraction}, we have
\begin{align}
    \mathcal{T}^P_{\eta}(\xi) &= \left(I - \frac{(\sgn(R_x(\eta)) \odot \lvert R_x(\eta)\rvert^{p-1})(\sgn(R_x(\eta)) \odot \lvert R_x(\eta)\rvert^{p-1})^T}{\|\lvert R_x(\eta)\rvert^{p-1}\|_2^2}\right)\xi\\
%    &= \left(I - \frac{(\sgn(x+\eta) \odot \lvert x+\eta\rvert^{p-1})(\sgn(x+\eta) \odot \lvert x+\eta\rvert^{p-1})^T}{\|\lvert x+\eta\rvert^{p-1}\|_2^2}\right)\xi\\
    &= \xi - \frac{(\sgn(x+\eta) \odot \lvert x+\eta\rvert^{p-1})^T \xi}{\|\lvert x+\eta\rvert^{p-1}\|_2^2}\sgn(x+\eta) \odot \lvert x+\eta\rvert^{p-1}.
\end{align}

\section{Summary of theoretical results}
\label{sec:summary}
We investigated the geometry of $S^{n-1}_p$ and proposed several retractions and their inverses and vector transports.
These results are summarized in Table~\ref{tab}.

\renewcommand{\arraystretch}{1.5}
\begin{table}[htbp]
\begin{center}
\caption{Summary of theoretical results. The sphere $S^{n-1}_p$ is defined for $p \in [1, \infty]$. However, the above results are for the case of~$p \in (1, \infty)$, where $S^{n-1}_p$ is a $C^1$ submanifold of~$\mathbb{R}^n$. In addition, we assume $x, \, y \in S^{n-1}_p$ and $\xi, \, \eta \in T_x S^{n-1}_p$.}\label{tab}%
\begin{tabular}{@{}ll@{}}
\toprule
Sphere with $p$-norm & $S^{n-1}_p = \{x \in \mathbb{R}^n \mid \|x\|_p = 1\}$.\\
\midrule
Riemannian metric on $S^{n-1}_p$ & $\langle \xi, \eta\rangle_x = \xi^T \eta$.\\
Induced norm in $T_x S^{n-1}_p$ & $\|\xi\|_x = \|\xi\|_2 = \sqrt{\xi^T \xi}$.\\
Tangent space at $x$   & $T_x S^{n-1}_p = \{\xi \in \mathbb{R}^n \mid \xi^T (\sgn(x) \odot \lvert x\rvert^{p-1}) = 0\}$.  \\
Normal space at $x$   & $N_x S^{n-1}_p = \{\alpha \sgn(x) \odot \lvert x\rvert^{p-1} \mid \alpha \in \mathbb{R}\}$.  \\
Orthogonal projection onto $T_x S^{n-1}_p$   & $P_x = I - \dfrac{(\sgn(x) \odot \lvert x\rvert^{p-1})(\sgn(x) \odot \lvert x\rvert^{p-1})^T}{\|\lvert x\rvert^{p-1}\|_2^2}$.\\
\midrule
Retraction by normalization & $R_x(\eta) = \dfrac{x+\eta}{\|x+\eta\|_p}$.\\
Inverse of $R_x$ & $R_x^{-1}(y) = \dfrac{y}{(\sgn(x) \odot \lvert x \rvert^{p-1})^T y} - x$,\\ & where $y$ satisfies $(\sgn(x) \odot \lvert x\rvert^{p-1})^T y > 0$.\\
Inverse of projective retraction & $(R^{\proj}_x)^{-1}(y) = y-x + \alpha\sgn(y) \odot \lvert y\rvert^{p-1}$,\\
& where $\alpha = \dfrac{1 - (\sgn(x) \odot \lvert x\rvert^{p-1})^T y}{(\sgn(x) \odot \lvert x\rvert^{p-1})^T(\sgn(y) \odot \lvert y\rvert^{p-1})}$\\
 & and $y$ is such that $\alpha \ge 0$.\\
Inverse of orthographic retraction & $(R^{\orth}_x)^{-1}(y) = y-x +\alpha  \sgn(x) \odot \lvert x\rvert^{p-1}$\\
& where $\alpha = \dfrac{1-(\sgn(x) \odot \lvert x\rvert^{p-1})^T y}{\|\lvert x\rvert^{p-1}\|_2^2}$\\
 &  and $y$ is such that $\alpha$ is the solution to \\
 & $\|x+\eta-\alpha\sgn(x) \odot \lvert x\rvert^{p-1}\|_p^p = 1$\\
  & with the smallest absolute value.\\
\midrule
Differentiated retraction of $R$ & 
$\phantom{=}\mathcal{T}^R_{\eta}(\xi) = \D R_x(\eta)[\xi]$ \\
& $= \dfrac{\xi}{\|z\|_p} - \dfrac{(\sgn(z) \odot \lvert z\rvert^{p-1}\big)^T \xi}{\|z\|_p^{p+1}}z$,\\
& where $z = x+\eta$.\\
Vector transport by projection &
$\phantom{=}\mathcal{T}^P_{\eta}(\xi) = P_{R_x(\eta)}(\xi)$\\
    & $= \xi - \dfrac{(\sgn(z) \odot \lvert z\rvert^{p-1})^T \xi}{\|\lvert z\rvert^{p-1}\|_2^2}\sgn(z) \odot \lvert z\rvert^{p-1}$,\\
& where $z = x+\eta$.\\
\hline
\end{tabular}
\end{center}
\end{table}

\section{Applications}
\label{sec:application}
In this section, we discuss two types of applications of $S^{n-1}_p$ for optimization.

\subsection{Nonnegative constraints on spheres}
In nonlinear optimization, we can introduce squared slack variables to handle nonnegative constraints~\cite{fukuda2017note}.
Specifically, the constraint $v \ge 0$ for $v \in \mathbb{R}^n$ is equivalent to $v = x^2$ with~$x \in \mathbb{R}^n$.
This idea can be used to address optimization problems on the sphere whose decision variable vector is constrained to be nonnegative.

\subsubsection{Unconstrained and constrained optimization problems on spheres with different norms}
For $p' \ge 1$ and $p = 2p'$, $S^{n-1}_p$ can be used to handle the variable on $S^{n-1}_{p'}$ with the nonnegative constraint.
To see this, we consider the following problem:
\begin{alignat*}{2}
&\text{minimize}&\quad & g(v)\\
&\text{subject to}&\quad & v \ge 0, \ v \in S^{n-1}_{p'},
\end{alignat*}
where $g \colon S^{n-1}_{p'} \to \mathbb{R}$ is the objective function.
This is a constrained Riemannian optimization problem on $S^{n-1}_{p'}$ with the constraint~$v \ge 0$.
Defining~\mbox{$v := x^2 \ge 0$} with $x = (x_i) \in \mathbb{R}^n$, we can observe that the conditions~$v \in S^{n-1}_{p'}$ and $v \ge 0$ are equivalent to~\mbox{$\|x^2\|_{p'} = 1$}.
Regarding the left-hand side, we have the relation~\mbox{$\|x^2\|_{p'}^{p'} = \sum_{i=1}^n \lvert x_i^2\rvert^{p'} = \sum_{i=1}^n \lvert x_i\rvert^{2p'} = \|x\|_{2p'}^{2p'} = \|x\|_p^p$}.
Therefore,~\mbox{$\|x^2\|_{p'} = 1$} is equivalent to $\|x\|_p = 1$, i.e., $x \in S^{n-1}_p$.
Hence, the aforementioned optimization problem is equivalent to the following problem:
\begin{alignat*}{2}
&\text{minimize}&\quad & f(x) := g(x^2)\\
&\text{subject to}&\quad & x \in S^{n-1}_{p},
\end{alignat*}
which is an unconstrained Riemannian optimization problem on $S^{n-1}_p$.

\subsubsection{Application to nonnegative PCA}
As a particular case of $p'= 2$ and $p = 4$, we can deduce from the above discussion that solving an optimization problem on $S^{n-1}_2$ with the nonnegative constraint on the decision variable vector is equivalent to solving the corresponding optimization problem on $S^{n-1}_4$ without constraint.
An important example within this framework is the nonnegative PCA~\cite{zass2006nonnegative}.

In~\cite{liu2020simple}, the nonnegative PCA is formulated as follows:
\begin{alignat}{2}
&\text{minimize}&\quad & -v^T A v\nonumber\\
&\text{subject to}&\quad & v \ge 0, \ v \in S^{n-1}_{2},\label{prob:sphere1}
\end{alignat}
where $A$ corresponds to the variance--covariance matrix of the data to be analyzed.
We assume that $A$ is an $n \times n$ symmetric positive definite matrix.
The above problem is equivalent to the following unconstrained problem on the sphere $S^{n-1}_4$ with $4$-norm:
\begin{alignat}{2}
&\text{minimize}&\quad & f(x) := -(x^2)^T A (x^2)\nonumber\\
&\text{subject to}&\quad & x \in S^{n-1}_{4}.\label{prob:sphere2}
\end{alignat}
For an optimal solution $x_*$ to the latter problem, $v_* := x_*^2$ is an optimal solution to the former problem.

We can further show that any critical point of $f$ in Problem~\eqref{prob:sphere2} satisfies the first-order optimality conditions for Problem~\eqref{prob:sphere1}.
First, we show that the KKT conditions are first-order necessary conditions for Problem~\eqref{prob:sphere1} and investigate the conditions.

\begin{proposition}
\label{prop:PCA_KKT}
Let $v_*$ be an optimal solution to Problem~\eqref{prob:sphere1} with an $n \times n$ symmetric positive definite matrix $A$.
Then, $v_*$ satisfies
\begin{equation}
\label{eq:KKT_v}
v_* \ge 0, \quad v_*^T v_* = 1, \quad (I - v_*v_*^T)Av_* \le 0.
\end{equation}
Specifically, the $i$th element $(Av_*)_i$ of $Av_*$ satisfies $(Av_*)_i = 0$ if $(v_*)_i > 0$, and~$(Av_*)_i \le 0$ if~$(v_*)_i = 0$.
In particular, if $v_* > 0$, then $Av_* = (v_*^T Av_*)v_*$ holds, i.e., $v_*^T Av_*$ and $v_*$ are an eigenvalue and associated eigenvector of $A$, respectively.
\end{proposition}

\begin{proof}
Problem~\eqref{prob:sphere1} is equivalent to the following Euclidean optimization problem:
\begin{alignat}{2}
&\text{minimize}&\quad & -v^T A v\nonumber\\
&\text{subject to}&\quad & v \ge 0, \ v^T v = 1, \ v \in \mathbb{R}^{n}.\label{prob:sphere3}
\end{alignat}
Throughout this proof, let~\mbox{$\mathcal{A}(v_*) := \{i_1, i_2, \dots, i_m\} \subset \{1, 2, \dots, n\}$} be the set of indices for the active inequality constraints at $v_*$ among the $n$ constraints~{$v_1 \ge 0$, $v_2 \ge 0$, \dots, $v_n \ge 0$}, and~$\bar{\mathcal{A}}(v_*) := \{1, 2, \dots, n\} - \mathcal{A}(v_*)$ be the complement of $\mathcal{A}(v_*)$ in $\{1, 2, \dots, n\}$, i.e., \begin{equation}
\label{eq:zeros}
    (v_*)_{i_1} = (v_*)_{i_2} = \dots = (v_*)_{i_m} = 0,
\end{equation}
and $(v_*)_i \neq 0$ for all $i \in \bar{\mathcal{A}}(v_*)$.
Then, letting $e_i \in \mathbb{R}^n$ be the vector whose $i$th element is $1$ and the others are $0$, the gradients of the $m$ functions defining the active inequality constraints are $e_{i_1}$, $e_{i_2}$, \dots, $e_{i_m}$.
Since $v_*^T v_* = 1$, $v_*$ is not $0$; and $\bar{\mathcal{A}}(v_*)$ is not empty, i.e, there exists~$i_0 \neq i_1, i_2, \dots, i_m$ such that $(v_*)_{i_0} \neq 0$.
Hence, the gradient of the equality constraint function $v^T v - 1$ at $v_*$, which is $2v_*$, and~\mbox{$e_{i_1}$, $e_{i_2}$, \dots, $e_{i_m}$} are linearly independent.
This means that the linear independent constraint qualification (LICQ)~\cite{nocedal2006numerical} holds at $v_*$, and the KKT conditions for~\eqref{prob:sphere3} are necessary optimality conditions.

Writing the KKT conditions explicitly, there exist $\lambda \in \mathbb{R}^n$ and $\mu \in \mathbb{R}$ such that
\begin{align}
    -2Av_* - \lambda + 2\mu v_* = 0,\label{eq:KKT1}\\*
    v_* \ge 0,\label{eq:v}\\*
    v_*^T v_* = 1,\label{eq:vv}\\*
    \lambda \ge 0,\label{eq:lambda}\\*
    \lambda \odot v_* = 0.\label{eq:complementary}
\end{align}
Under~\eqref{eq:v} and~\eqref{eq:lambda}, Eq.~\eqref{eq:complementary} is equivalent to the condition $\lambda^T v_* = 0$.
Using this and~\eqref{eq:vv}, and multiplying~\eqref{eq:KKT1} by $v_*^T$ from the left, we obtain $\mu = v_*^T Av_*$.
Therefore, \eqref{eq:KKT1} yields that~\mbox{$\lambda = 2((v_*^T Av_*)I-A)v_* \ge 0$}, which implies $(I-v_*v_*^T)Av_* \le 0$.
Thus, the conditions~\eqref{eq:KKT_v} are verified to hold.

Here, $I-v_*v_*^T$ is the orthogonal projection matrix to the orthogonal complement of the span of $v_* \ge 0$ with respect to the $2$-norm.
Therefore, from~\eqref{eq:zeros}, the intersection of the image~$\image(I-v_*v_*^T)$ of $I-v_*v_*^T$ and the nonpositive orthant~\mbox{$\mathbb{R}^n_- := \{x \in \mathbb{R}^n \mid x \le 0\}$} is
\begin{equation}
\label{eq:intersection}
    \image(I-v_*v_*^T) \cap \mathbb{R}^n_- = \{x = (x_i) \in \mathbb{R}^n_- \mid x_i = 0, \ i \in \bar{\mathcal{A}}(v_*)\}.
\end{equation}
It follows from~\eqref{eq:KKT_v} and~\eqref{eq:intersection} that the $i$th element of $(I-v_*v_*^T)Av_*$ is 
\begin{equation}
\label{eq:cases}
    ((I-v_*v_*^T)Av_*)_i
    \begin{cases}
        = 0 \quad \text{if $i \in \bar{\mathcal{A}}(v_*)$},\\
        \le 0 \quad \text{if $i \in \mathcal{A}(v_*)$}.
    \end{cases}
\end{equation}
Rewriting this, we obtain the relations $(Av_*)_i = (v_*^T Av_*)(v_*)_i$ if $i \in \bar{\mathcal{A}}(v_*)$, i.e., $(v_*)_i > 0$, and~\mbox{$(Av_*)_i \le (v_*^T Av_*)(v_*)_i = 0$} if $i \in \mathcal{A}(v_*)$, i.e., $(v_*)_i = 0$.

In particular, if $v_* > 0$, then $\mathcal{A}(v_*) = \emptyset$, and $(Av_*)_i = (v_*^T Av_*)(v_*)_i$ for all~\mbox{$i \in \{1, 2, \dots, n\}$}, which is equivalent to $Av_* = (v_*^T Av_*)v_*$.
This completes the proof.
\end{proof}

\begin{remark}
Conversely, it can be readily checked that if $v_* \in \mathbb{R}^n$ satisfies the conditions~\eqref{eq:KKT_v}, then $v_*$, $\lambda = 2((v_*^T Av_*)I-A)v_*$, and~\mbox{$\mu = v_*^T Av_*$} satisfy the KKT conditions~\eqref{eq:KKT1}--\eqref{eq:complementary}.
In summary, there exist $\lambda$ and $\mu$ such that the KKT conditions~\eqref{eq:KKT1}--\eqref{eq:complementary} are satisfied if and only if $v_*$ satisfies~\eqref{eq:KKT_v}.
\end{remark}

\begin{remark}
\label{rem:4}
From the last statement of Proposition~\ref{prop:PCA_KKT}, if there does not exist an eigenvector~$v$ associated with the largest eigenvalue of $A$ such that $v > 0$, then at least one inequality constraint is active at an optimal solution $v_*$ to Problem~\eqref{prob:sphere1}, i.e., $v_*$ contains at least one zero element.
\end{remark}

If $v_*$ is an optimal solution to Problem~\eqref{prob:sphere1}, $x_*$ satisfying $v_* = x_*^2$ is an optimal solution to Problem~\eqref{prob:sphere2}.
Therefore, such $x_*$ satisfies $\grad f(x_*) = 0$ on $S^{n-1}_4$.
More generally, as the following proposition claims, if $v_*$ satisfies the first-order necessary conditions~\eqref{eq:KKT_v} for Problem~\eqref{prob:sphere1}, then $x_*$ that satisfies~$v_* = x_*^2$ is a critical point of $f$ in~\eqref{prob:sphere2}.

\begin{proposition}
Consider Problem~\eqref{prob:sphere2} with an $n \times n$ symmetric positive definite matrix $A$.
The gradient of the objective function $f$ on $S^{n-1}_4$ satisfies
\begin{equation}
\label{eq:grad}
    \grad f(x) = -4\left((Ax^2)\odot x - \frac{(x^4)^T Ax^2}{\|x^3\|_2^2}x^3\right)
\end{equation}
for any $x \in S^{n-1}_4$.
Furthermore, if $v_* \in S^{n-1}_2$ satisfies~\eqref{eq:KKT_v} and $x_* \in S^{n-1}_4$ satisfies~$x_* = v_*^2$, then $\grad f(x_*) = 0$.
\end{proposition}

\begin{proof}
We first derive Eq.~\eqref{eq:grad} for $\grad f$.
Let $\bar{f}(x) := -(x^2)^T A (x^2)$ in $\mathbb{R}^n$, which is a smooth extension of $f$ to $\mathbb{R}^n$.
For any $d \in \mathbb{R}^n$, the directional derivative of $\bar{f}$ at $x$ in the direction of $d$ is computed as
\begin{equation}
    \D \bar{f}(x)[d] = -4(x^2)^T A(x \odot d) = -4 (Ax^2)^T (x\odot d) = -4 ((Ax^2) \odot x)^T d.
\end{equation}
Hence, we obtain $\nabla \bar{f}(x) = -4(Ax^2) \odot x$.
The Riemannian gradient $\grad f$ is then obtained by using the orthogonal projection~\eqref{eq:proj} as
\begin{align}
    \grad f(x) &= P_x(\nabla \bar{f}(x))\\
    &= -4 \left(I- \frac{(\sgn(x) \odot \lvert x\rvert^3)(\sgn(x) \odot \lvert x\rvert^3)^T}{\|x^3\|_2^2}\right)((A x^2) \odot x) \\
    &= -4\left((Ax^2)\odot x - \frac{(x^4)^T Ax^2}{\|x^3\|_2^2}x^3\right),
\end{align}
where we used $\sgn(x) \odot \lvert x\rvert^3 = x \odot \lvert x\rvert^2 = x^3$.
Thus, \eqref{eq:grad} is proved.

Subsequently, we assume that $v_*$ satisfies~\eqref{eq:KKT_v} and $x_*$ satisfies $v_* = x_*^2$.
As in the proof of Proposition~\ref{prop:PCA_KKT}, let $\mathcal{A}(v_*) = \{i_1, i_2, \dots, i_m\} \subset \{1, 2, \dots, n\}$ be the set of indices such that~\mbox{$(v_*)_{i_1} = (v_*)_{i_2} = \dots = (v_*)_{i_m} = 0$} holds and $\bar{\mathcal{A}}(v_*) := \{1, 2, \dots, n\} - \mathcal{A}(v_*)$.
Defining~\mbox{$\mu := v_*^T Av_*$}, it follows from~\eqref{eq:cases} that
\begin{equation}
\label{eq:v2}
    (v_*^2)^T Av_*
    % = \sum_{i=1}^n (v_*)_i^2 (Av_*)_i
    = \sum_{i \in \bar{\mathcal{A}}(v_*)} (v_*)_{i}^2 (Av_*)_{i}
    = \sum_{i \in \bar{\mathcal{A}}(v_*)} (v_*)_{i}^2 \mu (v_*)_{i}
    = \mu \sum_{i \in \bar{\mathcal{A}}(v_*)} (v_*)_{i}^3
    = \mu \|v_*\|_3^3.
\end{equation}
Here, from~\eqref{eq:cases}, we have $((I-v_*v_*^T)Av_*) \odot v_* = 0$, which, together with \mbox{$v_* \neq 0$} and~\eqref{eq:v2}, yields $(Av_*) \odot v_* = (v_*v_*^T Av_*) \odot v_*
= \mu v_*^2
= ((v_*^2)^T Av_* / \|v_*\|_3^3)v_*^2$.
% = \frac{(v_*^2)^T Av_*}{\|v_*\|_3^3}v_*^2$.
Substituting $v_* = x_*^2$, this is written as
\begin{equation}
\label{eq:24}
    (Ax_*^2) \odot x_*^2 = \frac{(x_*^4)^T A(x_*^2)}{\|x_*^2\|_3^3}x_*^4
    =\frac{(x_*^4)^T A(x_*^2)}{\|x_*^3\|_2^2}x_*^4.
\end{equation}
In general, for any $a, b, c \in \mathbb{R}$, $ac^2 = bc^4$ is equivalent to $ac = bc^3$.
Therefore, \eqref{eq:24} is reduced to
\begin{equation}
    (Ax_*^2) \odot x_* = \frac{(x_*^4)^T A(x_*^2)}{\|x_*^3\|_2^2}x_*^3,
\end{equation}
which shows that $\grad f(x_*) = 0$ in light of~\eqref{eq:grad}.
This completes the proof.
\end{proof}

\subsubsection{Numerical experiments for nonnegative PCA}
Here, we demonstrate numerical experiments for the nonnegative PCA.
To solve the constrained Problem~\eqref{prob:sphere1} on $S^{n-1}_2$,
we solve the unconstrained Problem~\eqref{prob:sphere2} on $S^{n-1}_4$ to obtain~\mbox{$x^{\proposed}_n \in S^{n-1}_4$}.
Then, we obtain~\mbox{$v^{\proposed}_n := (x^{\proposed}_n)^2$} as a solution to the original Problem~\eqref{prob:sphere1} based on the proposed framework.
For comparison, we also solve the constrained Euclidean optimization Problem~\eqref{prob:sphere3}, which is equivalent to Problem~\eqref{prob:sphere1}, using MATLAB's \texttt{fmincon} function, to obtain $v^{\fmincon}_n$.

We consider the two cases of $n = 10$ and $n = 1000$.
For each $n$, we constructed an $n \times n$ symmetric positive definite matrix $A$ with randomly generated elements.
Implementing the orthogonal projection~\eqref{eq:proj} and retraction~\eqref{eq:retraction} based on Manopt~\cite{boumal2014manopt}, we applied the Riemannian conjugate gradient method on $S^{n-1}_4$ to Problem~\eqref{prob:sphere2} with $n = 10$ and $n = 1000$.
The initial point $x_0$ for solving Problem~\eqref{prob:sphere2} was also randomly constructed, and we used $v_0 := x_0^2$ as the initial point for solving Problem~\eqref{prob:sphere3} by \texttt{fmincon}.

For $n = 10$, each elements of the two solutions~$v^{\proposed}_{10}$ and $v^{\fmincon}_{10}$ are the same to the third decimal place, as $(0.000,0.604,0.000,0.000,0.000,0.000,0.000,0.000,0.116,0.788)^T$.
Furthermore, they are sparse.
This is consistent with the discussion in Remark~\ref{rem:4}.

Subsequently, for $n = 1000$, we have~\mbox{$\|v^{\proposed}_{1000} - v^{\fmincon}_{1000}\|_2 = 1.307$}.
Furthermore, the values of the function~\mbox{$g(v) := -v^T Av$}, which should be minimized in Problems~\eqref{prob:sphere1} and~\eqref{prob:sphere3}, are~\mbox{$g(v^{\proposed}_{1000}) = -2.963 \times 10^{3}
    < -1.805 \times 10^{3}
    =g(v^{\fmincon}_{1000}).$}
% \begin{align*}
%     -(v^{\proposed}_{1000})^T A v^{\proposed}_{1000} &= -2.963 \times 10^{3}\\
%     &< -1.805 \times 10^{3}
%     =-(v^{\fmincon}_{1000})^T A v^{\fmincon}_{1000}.
% \end{align*}
In addition,~$v^{\proposed}_{1000}$ is sparse because $479$ of $1000$ elements of $v^{\proposed}_{1000}$ are less than $10^{-6}$, while no element of $v^{\fmincon}_{1000}$ is less than $10^{-6}$. 
Therefore, the proposed framework yielded a much better solution in this case.

\subsection{$L_p$-regularization-related optimization}
In certain applications, $L_p$ regularization is a frequently used technique, which considers an objective function as the weighted sum of the original objective function and the $p$-norm of the decision variable vector.
In particular, $L_1$ regularization is used in Lasso for sparse estimation~\cite{hastie2015statistical}.

\subsubsection{Relationship between regularized, constrained, and manifold optimization problems}
\label{subsubsec:relation}
We consider the following regularized optimization problem with~\mbox{$p \in [1, \infty]$}:
\begin{alignat}{2}
&\text{minimize}&\quad & L(w) + \lambda \|w\|_p\nonumber\\
&\text{subject to}&\quad & w \in \mathbb{R}^n,\label{prob:regularization}
\end{alignat}
where $L \colon \mathbb{R}^n \to \mathbb{R}$ is a convex function, and $\lambda \ge 0$ is a predefined constant called a regularization parameter.

Intuitively, $L_p$ regularization is closely related to considering the constraint that the $p$-norm of the decision variable vector is not larger than a predefined nonnegative constant.
Specifically, the corresponding constrained optimization problem is written as follows:
\begin{alignat}{2}
&\text{minimize}&\quad & L(w)\nonumber\\
&\text{subject to}&\quad & \|w\|_p \le C, \ w \in \mathbb{R}^n,\label{prob:constrained}
\end{alignat}
where $C \ge 0$ is a constant.
For example, while Lasso regression is performed by solving the former unconstrained Problem~\eqref{prob:regularization},
the latter Problem~\eqref{prob:constrained} is sometimes used to explain why 
Lasso tends to find a sparse solution~\cite{hastie2015statistical}.
This intuition is justified even for general $p \in [1, \infty]$ through the following proposition:

\begin{proposition}
\label{prop:regularization_constrained}
Assume that $p \in [1, \infty]$, and let $L \colon \mathbb{R}^n \to \mathbb{R}$ be a convex function.
If~$w_*$ is an optimal solution to Problem~\eqref{prob:regularization} with a predefined constant $\lambda \ge 0$,
then there exists $C \ge 0$ such that $w_*$ is an optimal solution to Problem~\eqref{prob:constrained} with $C$.
Conversely, if $w_*$ is an optimal solution to Problem~\eqref{prob:constrained} with a predefined constant~$C \ge 0$, then there exists $\lambda \ge 0$ such that~$w_*$ is an optimal solution to Problem~\eqref{prob:regularization}.
\end{proposition}

\begin{proof}
First, we fix $\lambda \ge 0$ and let $w_*$ be an optimal solution to Problem~\eqref{prob:regularization}.
Then, for any~$w \in \mathbb{R}^n$, we have
\begin{equation}
\label{eq:Lasso1}
L(w_*) + \lambda \|w_*\|_p \le L(w) + \lambda \|w\|_p.
\end{equation}
We show that $w_*$ is an optimal solution to Problem~\eqref{prob:constrained} with~\mbox{$C := \|w_*\|_p$}.
For any feasible solution $w \in \mathbb{R}^n$ to Problem~\eqref{prob:constrained}, we have
$\|w\|_p \le C = \|w_*\|_p$.
Combining this and~\eqref{eq:Lasso1}, we have $L(w_*) + \lambda \|w_*\|_p \le L(w) + \lambda \|w_*\|_p$, which means $L(w_*) \le L(w)$.
Furthermore, $w_*$ is clearly a feasible solution to Problem~\eqref{prob:constrained} since $\|w_*\|_p = C$.
Therefore, $w_*$ is an optimal solution to Problem~\eqref{prob:constrained}.

Conversely, we fix $C \ge 0$ and let $w_*$ be an optimal solution to Problem~\eqref{prob:constrained}.
Here, we additionally consider the Lagrange dual problem of~\eqref{prob:constrained}:
\begin{alignat}{2}
&\text{maximize}&\quad & \inf_{w \in \mathbb{R}^n} (L(w) + \mu(\|w\|_p - C))\nonumber\\
&\text{subject to}&\quad & \mu \ge 0, \ \mu \in \mathbb{R}.\label{prob:dual}
\end{alignat}
If $C > 0$, then Slater's condition for Problem~\eqref{prob:constrained}, which is that there exists~$w \in \mathbb{R}^n$ with~$\|w\|_p < C$, clearly holds with $w = 0$.
If $C = 0$, then the constraint $\|w\|_p \le C$ in Problem~\eqref{prob:constrained} is rewritten as the equality constraint $w = 0$, and Slater's condition (which in this case is that a feasible solution exists) holds by taking $w = 0$.
In each case, Slater's condition for Problem~\eqref{prob:constrained} holds.
Furthermore, Problem~\eqref{prob:constrained} is a convex optimization problem.
Therefore, it follows from Slater's theorem~\cite[Section 5.2.3]{boyd2004convex} that strong duality holds.
Hence, the optimal value $L(w_*)$ of Problem~\eqref{prob:constrained} and the optimal value of the dual problem~\eqref{prob:dual} coincide.
Letting~$\mu_* \ge 0$ be an optimal solution to~\eqref{prob:dual}, we have
\begin{equation}
\label{eq:Lasso3}
    L(w_*) = \inf_{w \in \mathbb{R}^n}(L(w) + \mu_*(\|w\|_p - C)).
\end{equation}
Since $\|w_*\|_p \le C$ and $\mu_* \ge 0$, we obtain $L(w_*) \le L(w_*) + \mu_*(\|w_*\|_p - C) \le L(w_*)$.
Thus,~$L(w_*) = L(w_*) + \mu_*(\|w_*\|_p - C)$ holds, and $w = w_*$ attains the minimum value of $L(w) + \mu_*(\|w\|_p - C)$ over all $w \in \mathbb{R}^n$.
Since $\mu_* C$ is a constant, $w = w_*$ also attains the minimum value of $L(w) + \mu_*\|w\|_p$ over $\mathbb{R}^n$.
This implies that $w_*$ is an optimal solution to Problem~\eqref{prob:regularization} with $\lambda = \mu_*$, thereby completing the proof.
\end{proof}

\begin{remark}
Although we focus on the $p$-norm here, Proposition~\ref{prop:regularization_constrained} can be straightforwardly generalized to the case with a general norm in $\mathbb{R}^n$.
Indeed, in the proof of Proposition~\ref{prop:regularization_constrained}, we do not exploit any specific property of the $p$-norm but properties of a general norm.
\end{remark}

From Proposition~\ref{prop:regularization_constrained}, we can observe the importance of Problem~\eqref{prob:constrained} in dealing with Problem~\eqref{prob:regularization}.
Furthermore, Problem~\eqref{prob:constrained} is closely related to the following problem:
\begin{alignat}{2}
&\text{minimize}&\quad & L(w)\nonumber\\*
&\text{subject to}&\quad & \|w\|_p = C, \ w \in \mathbb{R}^n,\label{prob:equality}
\end{alignat}
where $C$ is the constant in Problem~\eqref{prob:constrained}.
Indeed, if a minimum point $w_*$ of~$L$ over the entire~$\mathbb{R}^n$ lies in the ball $\{w \in \mathbb{R}^n \mid \|w\|_p \le C\}$, then $w_*$ is also an optimal solution to~\eqref{prob:constrained}.
Therefore, a practically more important case we focus on is when all minimum points of $L$ over $\mathbb{R}^n$ are outside the ball.
In this case, we can show that there exists an optimal solution to Problem~\eqref{prob:constrained} that is on the sphere $\{w \in \mathbb{R}^n \mid \|w\|_p = C\}$ as the following proposition:

\begin{proposition}
Assume $p \in [1, \infty]$, let $L \colon \mathbb{R}^n \to \mathbb{R}$ be convex, and consider Problem~\eqref{prob:constrained} with a constant $C \ge 0$.
Assume that any minimum point $y_*$ of $L$ over~$\mathbb{R}^n$ satisfies $\|y_*\|_p > C$.
Then, there exists an optimal solution $w_*$ to Problem~\eqref{prob:constrained} that satisfies~$\|w_*\|_p = C$.
\end{proposition}

\begin{proof}
Let $y_*$ and $z_*$ be a minimum point of $L$ over $\mathbb{R}^n$ and optimal solution to Problem~\eqref{prob:constrained}, respectively.
If $\|z_*\|_p = C$, then we can take $z_*$ as $w_*$ in the statement of the proposition.

In the remainder of the proof, we assume $\|z_*\| < C$.
From the assumption, we have $\|z_*\|_p < C < \|y_*\|_p$ and $L(y_*) \le L(z_*)$.
Since $L$ is convex, for~$\alpha \in [0, 1]$, we have
\begin{equation}
\label{eq:L}
    L(\alpha y_* + (1-\alpha)z_*) \le \alpha L(y_*) + (1-\alpha) L(z_*) \le \alpha L(z_*) + (1-\alpha) L(z_*) = L(z_*).
\end{equation}
Note that the function $\varphi(\alpha) := \|\alpha y_* + (1-\alpha)z_*\|_p$ is continuous with respect to~$\alpha$, where~$\varphi$ satisfies~\mbox{$\varphi(0) = \|z_*\|_p < C$} and $\varphi(1) = \|y_*\|_p > C$.
Therefore, from the intermediate value theorem, there exists $\alpha_* \in (0, 1)$ such that $\varphi(\alpha_*) = C$.
With this $\alpha_*$, defining~\mbox{$w_* := \alpha_* y_* + (1-\alpha_*)z_*$}, we have $\|w_*\|_p = \varphi(\alpha_*) = C$, implying that $w_*$ is feasible for Problem~\eqref{prob:constrained}.
Since $z_*$ is optimal for~\eqref{prob:constrained}, we have $L(z_*) \le L(w_*)$.
On the contrary, Eq.~\eqref{eq:L} yields that $L(w_*) \le L(z_*)$.
Thus, we obtain $L(w_*) = L(z_*)$, which means that $w_*$ is an optimal solution to Problem~\eqref{prob:constrained} with~$\|w_*\|_p = C$.
This completes the proof.
\end{proof}

From this proposition, if no minimum point of $L$ over $\mathbb{R}^n$ lies in the ball~\mbox{$\{w \in \mathbb{R}^n \mid \|w\|_p \le C\}$}, then any optimal solution to Problem~\eqref{prob:equality} is also an optimal solution to Problem~\eqref{prob:constrained}, i.e., it is sufficient to solve Problem~\eqref{prob:equality} for obtaining an optimal solution to Problem~\eqref{prob:constrained}.
Furthermore, upon scaling~$w \mapsto w/C$ and $L \mapsto L \circ CI$ and writing $w/C$ and $L \circ CI$ newly as $x$ and~$f$, respectively, i.e., $f(x) := L(Cx) = L(w)$, Problem~\eqref{prob:equality} essentially becomes equivalent to the following problem on the unit sphere $S^{n-1}_p$ with $p$-norm:
\begin{alignat}{2}
&\text{minimize}&\quad & f(x)\nonumber\\
&\text{subject to}&\quad & x \in S^{n-1}_p.\label{prob:sphere}
\end{alignat}

The important cases $p = 1$ and $p = \infty$ do not lie within the scope of the discussion in the previous sections.
Therefore, we approximate $S^{n-1}_1$ and $S^{n-1}_\infty$ by $S^{n-1}_p$ with $p = 1+\varepsilon$, where~$\varepsilon > 0$ is sufficiently small, and $S^{n-1}_p$ with sufficiently large $p$, respectively.

\subsubsection{Numerical experiment for Lasso regression}
Here, we consider the Lasso regression with simple artificial data.
The data size is set as~$m = 100$ and the number of variables as~\mbox{$n = 13$}.
We construct a data matrix $X \in \mathbb{R}^{m \times n}$ with randomly generated elements, set~\mbox{$w_* = (-5, -4, -3, -2, -1, 1, 2, 3, 4, 5, 0, 0, 0)^T \in \mathbb{R}^n$}, and compute~\mbox{$y = Xw_* + \epsilon$}, where each element of $\epsilon \in \mathbb{R}^n$ is randomly generated from a uniform distribution on the interval $[-1,1]$.
This means that, among $n = 13$ variables, the first $10$ are essential and the last $3$ have no effect in the data $y$.
We now estimate the coefficient parameter vector $w_*$ without any information on it, i.e., by using only the observed data $X$ and $y$.
An appropriate sparse estimation should yield a coefficient parameter vector whose last $3$ elements are close to $0$.

With the data $X$ and $y$, we consider Problem~\eqref{prob:equality} with $L(w) := \|Xw - y\|_2^2$ and a constant~$C > 0$, i.e., with the equality constraint $\|w\|_p = C$.
Note that we exclude the case when $C = 0$ since it yields the trivial solution $w = 0$.
Solving this problem is equivalent to minimizing $f(x) := L(Cx) = \|CXx - y\|_2^2$ with respect to $x \in S^{n-1}_p$, i.e., solving Problem~\eqref{prob:sphere}, and multiplying the resultant solution $x_*$ by $C$ to obtain the solution $w_* = Cx_*$ to Problem~\eqref{prob:equality}.

The case of $p = 1$ corresponds to the Lasso regression.
However, we can handle $S^{n-1}_p$ with $p > 1$ using the Riemannian optimization techniques developed in the previous sections.
Therefore, we adopt $p = 1.000001 = 1 + 10^{-6}$ and expect that solving the problem on $S^{n-1}_p$ yields a sparse solution.
Implementing the projection~\eqref{eq:proj} and retraction~\eqref{eq:retraction} based on Manopt, we applied the Riemannian conjugate gradient method for Problem~\eqref{prob:sphere} on~$S^{n-1}_p$.

In Table~\ref{tab2}, $w^{\nonreg} := (X^T X)^{-1}X^T y$ is the solution to the nonregularized optimization problem of minimizing $L$, i.e., Problem~\eqref{prob:regularization} with $\lambda = 0$.
As expected, this is not sparse.
Then, we applied the Riemannian conjugate gradient method in the proposed framework with several $C$ and obtained the solution $w^{\proposed}_C$ to Problem~\eqref{prob:equality} for each $C$.
The results for~\mbox{$C = 1, 5, 10, 20, 22, 25, 30, 50, 100$} are shown in the table.
For small $C$ such as~\mbox{$C = 1, 5, 10$}, the resultant solutions are sparse but do not provide a good estimation because the $5$th and $6$th entries are almost zero and the~$11$th and~$13$th are nonzero.
On the contrary, large $C$ does not contribute to sparse estimation at all.
Although finding the best value of $C$ is difficult, we observe that the case of $C = 22$ yields an appropriate solution in this experiment, which is a sparse solution with appropriate values.

For comparison, we also applied MATLAB's \texttt{lasso} function, which successively increases the value of $\lambda$ and solves Problem~\eqref{prob:regularization} for each $\lambda$.
For small~$\lambda$'s, the corresponding solutions are dense, whereas the solution is $0$ for a sufficiently large $\lambda$.
We focus on the $\lambda$'s and corresponding solutions $w^{\Lasso}_{\lambda} \in \mathbb{R}^n$ such that only the last $3$ elements of $w^{\Lasso}_{\lambda}$ are $0$.
The \texttt{lasso} function yielded several~$\lambda$'s satisfying this condition.
Among them, $w^{\Lasso}_{0.029}$ and $w^{\Lasso}_{0.746}$ correspond to the smallest and largest values of $\lambda$, respectively.
We can observe that $w^{\proposed}_{22}$ and $w^{\Lasso}_{0.746}$ are close to each other.

\begin{landscape}
\begin{table}
\begin{center}
\caption{Results obtained upon solving the Lasso-related optimization problems. The $i$th row shows the $i$th element of each solution.}\label{tab2}
\begin{tabular}{lccccccccccccc}
\toprule
 & $1$ & $2$ & $3$ & $4$ & $5$ & $6$ & $7$ & $8$ & $9$ & $10$ & $11$ & $12$ & $13$\\
\midrule
$w^{\nonreg}$ & $-5.055$ & $-3.904$ & $-3.022$ & $-2.039$ & $-1.036$ & $0.967$ & $1.972$ & $3.028$ & $4.036$ & $5.060$ & $-0.008$ & $-0.032$ & $0.052$ \\
\midrule
$w^{\proposed}_{1}$ & $-0.167$ & $-0.081$ & $-0.150$ & $-0.095$ & $0.000$ & $0.000$ & $0.137$ & $0.000$ & $0.048$ & $0.257$ & $0.005$ & $-0.000$ & $-0.059$\\
$w^{\proposed}_{5}$ & $-0.874$ & $-0.631$ & $-0.727$ & $-0.329$ & $-0.000$ & $0.003$ & $0.685$ & $0.044$ & $0.370$ & $1.334$ & $0.004$ & $-0.000$ & $-0.000$\\
$w^{\proposed}_{10}$ & $-1.683$ & $-1.335$ & $-1.359$ & $-0.709$ & $0.000$ & $0.002$ & $1.051$ & $0.537$ & $0.984$ & $2.268$ & $0.000$ & $0.000$ & $-0.071$\\
$w^{\proposed}_{20}$ & $-3.357$ & $-2.790$ & $-2.452$ & $-1.330$ & $-0.048$ & $0.255$ & $1.564$ & $1.765$ & $2.437$ & $3.907$ & $0.076$ & $-0.000$ & $-0.018$\\
$w^{\proposed}_{22}$ & $-3.779$ & $-3.234$ & $-2.537$ & $-1.202$ & $-0.119$ & $0.294$ & $1.819$ & $1.914$ & $2.829$ & $4.272$ & $0.000$ & $0.000$ & $-0.000$\\
$w^{\proposed}_{25}$ & $-4.193$ & $-3.422$ & $-2.791$ & $-1.599$ & $-0.510$ & $0.587$ & $1.787$ & $2.391$ & $3.203$ & $4.504$ & $0.008$ & $0.004$ & $0.000$\\
$w^{\proposed}_{30}$ & $-5.027$ & $-3.895$ & $-3.014$ & $-2.016$ & $-1.021$ & $0.952$ & $1.967$ & $3.010$ & $4.012$ & $5.042$ & $0.000$ & $-0.013$ & $0.030$\\
$w^{\proposed}_{50}$ & $-8.040$ & $-5.762$ & $-4.521$ & $-3.657$ & $-2.872$ & $-0.979$ & $-0.608$ & $5.474$ & $6.346$ & $7.002$ & $-2.290$ & $0.979$ & $-1.471$\\
$w^{\proposed}_{100}$ & $-15.03$ & $0.127$ & $-10.70$ & $-9.352$ & $-5.748$ & $-7.317$ & $-7.698$ & $11.94$ & $0.913$ & $12.70$ & $-7.959$ & $4.992$ & $-5.526$\\
\midrule
$w^{\Lasso}_{0.029}$ & $-4.989$ & $-3.881$ & $-3.023$ & $-1.986$ & $-0.993$ & $0.939$ & $1.959$ & $2.976$ & $3.981$ & $5.037$ & $0$ & $0$ & $0$\\
$w^{\Lasso}_{0.746}$ & $-3.727$ & $-3.212$ & $-2.712$ & $-1.192$ & $-0.002$ & $0.224$ & $1.831$ & $1.890$ & $2.791$ & $4.314$ & $0$ & $0$ & $0$\\
\end{tabular}
\end{center}
\end{table}
\end{landscape}

\subsubsection{Numerical experiment for box-constrained problem}
Here, we consider the following box-constrained optimization problem:
\begin{alignat}{2}
&\text{minimize}&\quad & L(w)\nonumber\\
&\text{subject to}&\quad & l \le w \le u, \ w \in \mathbb{R}^n,\label{prob:box}
\end{alignat}
where $l = (l_i),\, u = (u_i) \in \mathbb{R}^n$ are given constant vectors with $l < u$.\footnote{If $l_i = u_i$ for some $i$, then the constant $l_i$ is the only value that the corresponding $w_i$ can take. By eliminating such a constant variable in advance if necessary, we can assume $l < u$ without loss of generality.}
The constraint $l \le w \le u$ means the box constraint $l_i \le w_i \le u_i$ for~$i = 1, 2, \dots, n$.
Defining $a := (u-l) / 2 > 0$ and $b:= (l+u)/2$, this constraint is rewritten as $-a \le w - b \le a$, which is equivalent to~\mbox{$-\bm{1} \le D^{-1}(w-b) \le \bm{1}$}, i.e.,~\mbox{$\|D^{-1}(w-b)\|_{\infty} \le 1$}, with $D$ being the $n \times n$ diagonal matrix with diagonal elements $a_1, a_2, \dots, a_n > 0$.
Therefore, with the transformation~\mbox{$x := D^{-1}(w-b) \in S^{n-1}_{\infty}$} and~\mbox{$f(x) := L(a \odot x + b) = L(Dx + b) = L(w)$}, solving Problem~\eqref{prob:box} is essentially equivalent to minimizing $f$ in the unit ball~$B^n_{\infty} = \{x \in \mathbb{R}^n \mid \|x\|_{\infty} \le 1\}$.
Consider a practical case where no minimum point of $f$ over the entire $\mathbb{R}^n$ is in the ball~$B^n_{\infty}$.
Then, as discussed in Section~\ref{subsubsec:relation}, we only have to solve Problem~\eqref{prob:sphere} on the sphere $S^{n-1}_p$ with~$p = \infty$.
However, since $p = \infty$ was excluded from the discussion in the previous sections, we instead need to consider a sufficiently large finite value~$p$ when solving the problem numerically.

We performed a numerical experiment for the following problem with $n = 10$:
\begin{alignat}{2}
&\text{minimize}&\quad & L(w) := \frac{1}{2}w^T A w + c^T w\nonumber\\
&\text{subject to}&\quad & l \le w \le u, \ w \in \mathbb{R}^n,\label{prob:box_example}
\end{alignat}
where the elements of the $n \times n$ symmetric positive definite matrix $A$ and vector~$c \in \mathbb{R}^n$ are randomly generated.
We set $l = (-1, -2, \dots, -10)^T$ and~$u = (1, 2, \dots, 10)^T$.
Note that~\mbox{$\nabla L(w) = Aw + c$} and the minimum point of $L$ over the entire $\mathbb{R}^n$ is $-A^{-1}c$.
We checked that~$w^{\unconst} := -A^{-1}c$ is not feasible for Problem~\eqref{prob:box_example} in this case.
Therefore, as discussed above, if we minimize $f(x) := L(a \odot x +b)$ with $a := (u-l)/2$ and~\mbox{$b := (l+u)/2$} on the sphere~$S^{n-1}_{\infty}$ to obtain $x_*$,
then $w_* := a \odot x_*+b$ is an optimal solution to Problem~\eqref{prob:box_example}.
We approximated~$S^{n-1}_{\infty}$ by $S^{n-1}_p$ with~$p = 5, 10, 50, 100, 500, 1000, 5000, 10000, 50000$, and solved Problem~\eqref{prob:sphere} by the Riemannian conjugate gradient method based on Manopt.
We denote the resultant approximate solution to the original Problem~\eqref{prob:box_example} by $w^{\proposed}_p$ for each $p$ and the solution to Problem~\eqref{prob:box_example} obtained using MATLAB's \texttt{fmincon} function by $w^{\fmincon}$.
The results are shown in Table~\ref{tab:3}.
As expected, the larger the value of $p$, the more accurate is the obtained solution.

\begin{landscape}
\begin{table}
\begin{center}
\caption{Results of solving box-constrained-optimization-related problems. The $i$th row shows the $i$th element of each solution, and the rightmost column shows the distance between the resultant vectors and $w^{\fmincon}$.}\label{tab:3}
\begin{tabular}{lccccccccccc}
\toprule
 & $1$ & $2$ & $3$ & $4$ & $5$ & $6$ & $7$ & $8$ & $9$ & $10$ & $\|w - w^{\fmincon}\|_2$\\
\midrule
$w^{\unconst}$ & $-3.335$ & $4.331$ & $-1.575$ & $-0.383$ & $-1.127$ & $5.731$ & $-3.268$ & $-0.024$ & $2.072$ & $-1.885$ & $5.971$\\
\midrule
$w^{\proposed}_{5}$ & $-0.725$ & $1.912$ & $-0.629$ & $-0.187$ & $-0.693$ & $1.697$ & $-0.863$ & $-0.011$ & $0.500$ & $-0.570$ & $0.4488$\\
$w^{\proposed}_{10}$ & $-0.855$ & $1.954$ & $-0.644$ & $-0.243$ & $-0.728$ & $1.797$ & $-0.931$ & $0.008$ & $0.576$ & $-0.536$ & $0.2357$\\
$w^{\proposed}_{50}$ & $-0.969$ & $1.991$ & $-0.657$ & $-0.294$ & $-0.761$ & $1.882$ & $-0.989$ & $0.026$ & $0.643$ & $-0.503$ & $4.952 \times 10^{-2}$\\
$w^{\proposed}_{100}$ & $-0.985$ & $1.995$ & $-0.658$ & $-0.301$ & $-0.765$ & $1.893$ & $-0.997$ & $0.028$ & $0.652$ & $-0.499$ & $2.493 \times 10^{-2}$\\
$w^{\proposed}_{500}$ & $-0.997$ & $1.999$ & $-0.659$ & $-0.306$ & $-0.769$ & $1.902$ & $-1.003$ & $0.030$ & $0.659$ & $-0.495$ & $5.012 \times 10^{-3}$\\
$w^{\proposed}_{1000}$ & $-0.998$ & $2.000$ & $-0.660$ & $-0.307$ & $-0.769$ & $1.903$ & $-1.004$ & $0.030$ & $0.660$ & $-0.494$ & $2.508 \times 10^{-3}$\\
$w^{\proposed}_{5000}$ & $-1.000$ & $2.000$ & $-0.660$ & $-0.308$ & $-0.770$ & $1.904$ & $-1.004$ & $0.030$ & $0.660$ & $-0.494$ & $5.014 \times 10^{-4}$\\
$w^{\proposed}_{10000}$ & $-1.000$ & $2.000$ & $-0.660$ & $-0.308$ & $-0.770$ & $1.904$ & $-1.004$ & $0.030$ & $0.660$ & $-0.494$ & $2.508 \times 10^{-4}$\\
$w^{\proposed}_{50000}$ & $-1.000$ & $2.000$ & $-0.660$ & $-0.308$ & $-0.770$ & $1.904$ & $-1.004$ & $0.031$ & $0.660$ & $-0.494$ & $5.030 \times 10^{-5}$\\
\midrule
$w^{\fmincon}$ & $-1.000$ & $2.000$ & $-0.660$ & $-0.308$ & $-0.770$ & $1.904$ & $-1.004$ & $0.031$ & $0.660$ & $-0.494$ & $0$
\end{tabular}
\end{center}
\end{table}
\end{landscape}

\section{Concluding remarks}
\label{sec:conclusion}
In this paper, we investigated the geometry of the unit sphere defined via the $p$-norm as $S^{n-1}_p := \{x \in \mathbb{R}^n \mid \|x\|_p = 1\}$ with $p \in [1, \infty]$, especially $p \in (1, \infty)$, in detail.
In particular, we derived formulas for retractions, their inverses, and vector transports, which can be used in Riemannian optimization algorithms.
The results are summarized in Table~\ref{tab} of Section~\ref{sec:summary}.

Furthermore, we discussed two types of applications of optimization on~$S^{n-1}_p$.
The first was for optimization problems on the sphere with the nonnegative constraint, which include the nonnegative PCA problem.
The second was for $L_p$-regularization-related optimization problems, which are closely related to the Lasso regression and box-constrained problems.
To this end, we provided mathematical support for the applications and performed numerical experiments to verify the validity of the theory.

The applications addressed in this paper are examples of the proposed theory, and the corresponding numerical experiments are preliminary ones.
Therefore, developing more efficient algorithms by combining the present theory and existing Riemannian optimization theory than state-of-the-art algorithms for specific problems, e.g., the nonnegative PCA and Lasso problems, are left for future work.

\bibliographystyle{abbrv}
\bibliography{sato_springer}   % name your BibTeX data base

\end{document}